\documentclass[11pt,twoside]{article}
\usepackage{amsmath, amsthm, amscd, amsfonts, amssymb, graphicx, color}

\date{}

\begin{document}

\centerline{}

\centerline {\Large{\bf Generalized atomic subspaces for operators in Hilbert spaces }}

%% My definition
\newcommand{\mvec}[1]{\mbox{\bfseries\itshape #1}}
\centerline{}
\centerline{\textbf{Prasenjit Ghosh}}
\centerline{Department of Pure Mathematics, University of Calcutta,}
\centerline{35, Ballygunge Circular Road, Kolkata, 700019, West Bengal, India}
\centerline{e-mail: prasenjitpuremath@gmail.com}
\centerline{}
\centerline{\textbf{T. K. Samanta}}
\centerline{Department of Mathematics, Uluberia College,}
\centerline{Uluberia, Howrah, 711315,  West Bengal, India}
\centerline{e-mail: mumpu$_{-}$tapas5@yahoo.co.in}

\newtheorem{Theorem}{\quad Theorem}[section]

\newtheorem{definition}[Theorem]{\quad Definition}

\newtheorem{theorem}[Theorem]{\quad Theorem}

\newtheorem{remark}[Theorem]{\quad Remark}

\newtheorem{corollary}[Theorem]{\quad Corollary}

\newtheorem{note}[Theorem]{\quad Note}

\newtheorem{lemma}[Theorem]{\quad Lemma}

\newtheorem{example}[Theorem]{\quad Example}

\newtheorem{result}[Theorem]{\quad Result}
\newtheorem{conclusion}[Theorem]{\quad Conclusion}

\newtheorem{proposition}[Theorem]{\quad Proposition}

\centerline{}

\begin{abstract}
\textbf{\emph{We introduce the notion of a g-atomic subspace for a bounded linear operator and construct several useful resolutions of the identity operator on a Hilbert space using the theory of g-fusion frames.\,Also we shall describe the concept of frame operator for a pair of g-fusion Bessel sequences and some of their properties.}}
\end{abstract}



%=====================================
\section{Introduction}
%=====================================
 
\smallskip\hspace{.6 cm} Frames for Hilbert spaces were first introduced by Duffin and Schaeffer \cite{Duffin} in 1952 to study some fundamental problems in non-harmonic Fourier series.\;Later on, after some decades, frame theory was popularized by Daubechies, Grossman, Meyer \cite{Daubechies}.\;At present, frame theory has been widely used in signal and image processing, filter bank theory, coding and communications, system modeling and so on.\;Several generalizations of frames  namely, \,$K$-frames, \,$g$-frames, fusion frames etc. have been introduced in recent times.

$K$-frames were introduced by L.\,Gavruta \cite{L} to study the atomic system with respect to a bounded linear operator.\;Using frame theory techiques, the author also studied the atomic decompositions for operators on reporducing kernel Hilbert spaces \cite{G}.\;Sun \cite{Sun}\, introduced a \,$g$-frame and a \,$g$-Riesz basis in complex Hilbert spaces and discussed several properties of them.\;Huang \cite{Hua} began to study \,$K$-$g$-frame by combining \,$K$-frame and \,$g$-frame.\;P.\,Casazza \cite{Kutyniok} was first to introduce the notion of fusion frames or frames of subspaces and gave various ways to obtain a resolution of the identity operator from a fuison frame.\;The concept of an atomic subspace with respect to a bounded linear operator were introduced by A.\,Bhandari and S.\,Mukherjee \cite{Bhandari}.\;Construction of \,$K$-$g$-fusion frames and their dual were presented by Sadri and Rahimi \cite{Sadri} to generalize the theory of \,$K$-frame, fusion frame and \,$g$-frame.\,P. Ghosh and T. K. Samanta \cite{P} studied the stability of dual \,$g$-fusion frames in Hilbert spaces. 

In this paper,\;we present some useful results about resolution of the identity operator on a Hilbert space using the theory of \,$g$-fusion frames.\;We give the notion of \,$g$-atomic subspace with respect to a bounded linear operator.\;The frame operator for a pair of \,$g$-fusion Bessel sequences are discussed and some properties are going to be established.

The paper is organized as follows; in Section 2, we briefly recall the basic definitions and results.\;Various ways of obtaining resolution of the identity operator on a Hilbert space in \,$g$-fusion frame are studied in Section 3.\;$g$-atomic subspaces are introduced and discussed in Section 4.\,In Section 5, frame operator for a pair of \,$g$-fusion Bessel sequences are given and establish various properties.

Throughout this paper,\;$H$\; is considered to be a separable Hilbert space with associated inner product \,$\left <\,\cdot \,,\, \cdot\,\right>$\, and \,$\left\{\,H_{j}\,\right\}_{ j \,\in\, J}$\, are the collection of Hilbert spaces, where \,$J$\; is subset of  integers \,$\mathbb{Z}$.\;$I_{H}$\; is the identity operator on \,$H$.\;$\mathcal{B}\,(\,H_{\,1},\, H_{\,2}\,)$\; is a collection of all bounded linear operators from \,$H_{\,1} \,\text{to}\, H_{\,2}$.\;In particular \,$\mathcal{B}\,(\,H\,)$\; denote the space of all bounded linear operators on \,$H$.\;For \,$T \,\in\, \mathcal{B}\,(\,H\,)$, we denote \,$\mathcal{N}\,(\,T\,)$\; and \,$\mathcal{R}\,(\,T\,)$\; for null space and range of \,$T$, respectively.\,Also, \,$P_{\,V} \,\in\, \mathcal{B}\,(\,H\,)$\; is the orthonormal projection onto a closed subspace \,$V \,\subset\, H$.\;Define the space
\[l^{\,2}\left(\,\left\{\,H_{j}\,\right\}_{ j \,\in\, J}\,\right) \,=\, \left \{\,\{\,f_{\,j}\,\}_{j \,\in\, J} \,:\, f_{\,j} \;\in\; H_{j},\; \sum\limits_{\,j \,\in\, J}\, \left \|\,f_{\,j}\,\right \|^{\,2} \,<\, \infty \,\right\}\]
with inner product is given by \,$\left<\,\{\,f_{\,j}\,\}_{ j \,\in\, J} \,,\, \{\,g_{\,j}\,\}_{ j \,\in\, J}\,\right> \;=\; \sum\limits_{\,j \,\in\, J}\, \left<\,f_{\,j} \,,\, g_{\,j}\,\right>_{H_{j}}$.\,Clearly \,$l^{\,2}\left(\,\left\{\,H_{j}\,\right\}_{ j \,\in\, J}\,\right)$\; is a Hilbert space with the pointwise operations \cite{Sadri}.

%=====================================
\section{Preliminaries}
%=====================================

\begin{theorem}(\,Douglas' factorization theorem\,)\,{\cite{Douglas}}\label{th1}
Let \;$U,\, V \,\in\, \mathcal{B}\,(\,H\,)$.\,Then the following conditions are equivalent:
\begin{itemize}
\item[(\,1\,)]\;\;$\mathcal{R}\,(\,U\,) \,\subseteq\, \mathcal{R}\,(\,V\,)$.
\item[(\,2\,)]\;\;$U\, U^{\,\ast} \,\leq\, \lambda^{\,2}\; V\,V^{\,\ast}$\; for some \,$\lambda \,>\, 0$.
\item[(\,3\,)]\;\;$U \,=\, V\,W$\, for some bounded linear operator \,$W$\, on \,$H$.
\end{itemize}
\end{theorem}

\begin{theorem}\cite{Jain}\label{th1.001}
The set \,$\mathcal{S}\,(\,H\,)$\; of all self-adjoint operators on \,$H$\; is a partially ordered set with respect to the partial order \,$\leq$\, which is defined as for \,$T,\,S \,\in\, \mathcal{S}\,(\,H\,)$ 
\[T \,\leq\, S \,\Leftrightarrow\, \left<\,T\,f \,,\, f\,\right> \,\leq\, \left<\,S\,f \,,\, f\,\right>\; \;\forall\; f \,\in\, H.\] 
\end{theorem}

\begin{theorem}\cite{Gavruta}\label{th1.01}
Let \,$V \,\subset\, H$\; be a closed subspace and \,$T \,\in\, \mathcal{B}\,(\,H\,)$.\;Then \,$P_{\,V}\, T^{\,\ast} \,=\, P_{\,V}\,T^{\,\ast}\, P_{\,\overline{T\,V}}$.\;If \,$T$\; is an unitary operator (\,i\,.\,e \,$T^{\,\ast}\, T \,=\, I_{H}$\,), then \,$P_{\,\overline{T\,V}}\;T \,=\, T\,P_{\,V}$.
\end{theorem}

\begin{definition}\cite{O}
A sequence \,$\left\{\,f_{\,j}\,\right\}_{j \,\in\, J}$\, of elements in \,$H$\, is a frame for \,$H$\, if there exist constants \,$A,\, B \,>\, 0$\, such that
\[ A\, \|\,f\,\|^{\,2} \,\leq\, \sum\limits_{j \,\in\, J}\, \left|\, \left<\,f \,,\, f_{\,j} \, \right>\,\right|^{\,2} \,\leq\, B \,\|\,f\,\|^{\,2}\; \;\forall\; f \,\in\, H. \]
The constants \,$A$\, and \,$B$\, are called frame bounds.
\end{definition}

\begin{definition}\cite{Kutyniok}
Let \,$\left\{\,W_{j}\,\right\}_{ j \,\in\, J}$\; be a collection of closed subspaces of \,$H$\; and \,$\left\{\,v_{j}\,\right\}_{ j \,\in\, J}$\, be a collection of positive weights.\;A family of weighted closed subspaces \,$\left\{\, (\,W_{j},\, v_{j}\,) \,:\, j \,\in\, J\,\right\}$\; is called a fusion frame for \,$H$\; if there exist constants \,$0 \,<\, A \,\leq\, B \,<\, \infty$\; such that
\[A \;\left\|\,f \,\right\|^{\,2} \,\leq\, \sum\limits_{\,j \,\in\, J}\, v_{j}^{\,2}\,\left\|\,P_{\,W_{j}}\,(\,f\,) \,\right\|^{\,2} \,\leq\, B \; \left\|\,f \, \right\|^{\,2}\; \;\forall\; f \,\in\, H.\]
The constants \,$A,\, B$\; are called fusion frame bounds.\;If \,$A \,=\, B$\; then the fusion frame is called a tight fusion frame, if \,$A \,=\, B \,=\, 1$\; then it is called a Parseval fusion frame.
\end{definition}

\begin{definition}\cite{Bhandari}
Let \,$\left\{\,W_{j}\,\right\}_{ j \,\in\, J}$\, be a family of closed subspaces of \,$H$\, and \,$\left\{\,v_{j}\,\right\}_{ j \,\in\, J}$\, be a family of positive weights and \,$K \,\in\, \mathcal{B}\,(\,H\,)$.\;Then \,$\left\{\, (\,W_{j},\, v_{j}\,) \,:\, j \,\in\, J\,\right\}$\, is said to be an atomic subspace of \,$H$\, with respect to \,$K$\, if following conditions hold:
\begin{itemize}
\item[$(I)$] \,$\sum\limits_{\,j \,\in\, J}\,v_{j}\,f_{j}$\, is convergent for all \,$\left\{\,f_{\,j}\,\right\}_{j \,\in\, J} \,\in\, \left(\,\sum\limits_{\,j \,\in\, J}\,\oplus\, W_{j}\,\right)_{l^{\,2}}$.
\item[(II)] For every \,$f \,\in\, H$, there exists \,$\left\{\,f_{\,j}\,\right\}_{j \,\in\, J} \,\in\, \left(\,\sum\limits_{\,j \,\in\, J}\,\oplus\, W_{j}\,\right)_{l^{\,2}}$\, such that 
\[K\,(\,f\,) \,=\, \sum\limits_{\,j \,\in\, J}\, v_{j}\,f_{\,j}\; \;\text{and}\; \;\left\|\,\left\{\,f_{\,j}\,\right\}\,\right\|_{\left(\,\sum\limits_{\,j \,\in\, J}\,\oplus\, W_{j}\,\right)_{l^{\,2}}} \,\leq\, C\; \|\,f\,\|_{H}\]\; for some \,$C \,>\, 0$, where 
\[\left(\,\sum\limits_{\,j \,\in\, J}\,\oplus\, W_{j}\,\right)_{l^{\,2}} \,=\, \left \{\,\{\,f_{\,j}\,\}_{j \,\in\, J} \,:\, f_{\,j} \,\in\, W_{j},\, \sum\limits_{\,j \,\in\, J}\, \left \|\,f_{\,j}\,\right \|^{\,2} \,<\, \infty \,\right\}\]
with inner product is given by \;$\left<\,\{\,f_{\,j}\,\}_{ j \,\in\, J} \,,\, \{\,g_{\,j}\,\}_{ j \,\in\, J}\,\right> \,=\, \sum\limits_{\,j \,\in\, J}\, \left<\,f_{\,j} \,,\, g_{\,j}\,\right>_{H}$. 
\end{itemize}    
\end{definition}

\begin{definition}\cite{Sun}
A sequence \,$\left\{\,\Lambda_{j} \,\in\, \mathcal{B}\,(\,H,\, H_{j}\,) \,:\, j \,\in\, J\,\right\}$\, is called a generalized frame or g-frame for \,$H$\; with respect to \,$\left\{\,H_{j}\,\right\}_{j \,\in\, J}$\; if there are two positive constants \,$A$\, and \,$B$\, such that
\[A \;\left \|\, f \,\right \|^{\,2} \,\leq\, \sum\limits_{\,j \,\in\, J}\, \left\|\,\Lambda_{j}\,f \,\right\|^{\,2} \,\leq\, B \; \left\|\, f \, \right\|^{\,2}\; \;\forall\; f \,\in\, H.\]
The constants \,$A$\; and \,$B$\; are called the lower and upper frame bounds, respectively.
\end{definition}

\begin{definition}\cite{Ahmadi,Sadri}
Let \,$\left\{\,W_{j}\,\right\}_{ j \,\in\, J}$\, be a collection of closed subspaces of \,$H$\; and \,$\left\{\,v_{j}\,\right\}_{ j \,\in\, J}$\; be a collection of positive weights and let \,$\Lambda_{j} \,\in\, \mathcal{B}\,(\,H,\, H_{j}\,)$\; for each \,$j \,\in\, J$.\;Then the family \,$\Lambda \,=\, \{\,\left(\,W_{j},\, \Lambda_{j},\, v_{j}\,\right)\,\}_{j \,\in\, J}$\; is called a generalized fusion frame or a g-fusion frame for \,$H$\; respect to \,$\left\{\,H_{j}\,\right\}_{j \,\in\, J}$\; if there exist constants \,$0 \,<\, A \,\leq\, B \,<\, \infty$\; such that
\begin{equation}\label{eq1}
A \;\left \|\,f \,\right \|^{\,2} \,\leq\, \sum\limits_{\,j \,\in\, J}\,v_{j}^{\,2}\, \left\|\,\Lambda_{j}\,P_{\,W_{j}}\,(\,f\,) \,\right\|^{\,2} \,\leq\, B \; \left\|\, f \, \right\|^{\,2}\; \;\forall\; f \,\in\, H.
\end{equation}
The constants \,$A$\; and \,$B$\; are called the lower and upper bounds of g-fusion frame, respectively.\,If \,$A \,=\, B$\; then \,$\Lambda$\; is called tight g-fusion frame and if \;$A \,=\, B \,=\, 1$\, then we say \,$\Lambda$\; is a Parseval g-fusion frame.\;If  \,$\Lambda$\; satisfies only the condition
\[\sum\limits_{\,j \,\in\, J}\, v_{j}^{\,2}\; \left\|\,\Lambda_{j}\, P_{\,W_{j}}\,(\,f\,) \,\right\|^{\,2} \,\leq\, B \; \left\|\, f \, \right\|^{\,2}\; \;\forall\; f \,\in\, H\]then it is called a g-fusion Bessel sequence with bound \,$B$\; in \,$H$. 
\end{definition}

\begin{definition}\cite{Sadri}
Let \,$\Lambda \,=\, \left\{\,\left(\,W_{j},\, \Lambda_{j},\, v_{j}\,\right)\,\right\}_{j \,\in\, J}$\, be a g-fusion Bessel sequence in \,$H$\, with a bound \,$B$.\;The synthesis operator \,$T_{\Lambda}$\, of \,$\Lambda$\; is defined as 
\[ T_{\Lambda} \,:\, l^{\,2}\left(\,\left\{\,H_{j}\,\right\}_{ j \,\in\, J}\,\right) \,\to\, H,\]
\[T_{\Lambda}\,\left(\,\left\{\,f_{\,j}\,\right\}_{j \,\in\, J}\,\right) \,=\,  \sum\limits_{\,j \,\in\, J}\, v_{j}\, P_{\,W_{j}}\,\Lambda_{j}^{\,\ast}\,f_{j}\; \;\;\forall\; \{\,f_{j}\,\}_{j \,\in\, J} \,\in\, l^{\,2}\left(\,\left\{\,H_{j}\,\right\}_{ j \,\in\, J}\,\right)\] and the analysis operator is given by 
\[ T_{\Lambda}^{\,\ast} \,:\, H \,\to\, l^{\,2}\left(\,\left\{\,H_{j}\,\right\}_{ j \,\in\, J}\,\right),\; T_{\Lambda}^{\,\ast}\,(\,f\,) \,=\,  \left\{\,v_{j}\,\Lambda_{j}\, P_{\,W_{j}}\,(\,f\,)\,\right\}_{ j \,\in\, J}\; \;\forall\; f \,\in\, H.\]
The g-fusion frame operator \,$S_{\Lambda} \,:\, H \,\to\, H$\; is defined as follows:
\[S_{\Lambda}\,(\,f\,) \,=\, T_{\Lambda}\,T_{\Lambda}^{\,\ast}\,(\,f\,) \,=\, \sum\limits_{\,j \,\in\, J}\, v_{j}^{\,2}\; P_{\,W_{j}}\, \Lambda_{j}^{\,\ast}\; \Lambda_{j}\, P_{\,W_{j}}\,(\,f\,)\]and it can be easily verify that 
\[\left<\,S_{\Lambda}\,(\,f\,) \,,\, f\,\right> \,=\, \sum\limits_{\,j \,\in\, J}\, v_{j}^{\,2}\, \left\|\,\Lambda_{j}\, P_{\,W_{j}}\,(\,f\,) \,\right\|^{\,2}\; \;\forall\; f \,\in\, H.\]
Furthermore, if \,$\Lambda$\, is a g-fusion frame with bounds \,$A$\, and \,$B$\, then from (\ref{eq1}),
\[\left<\,A\,f \,,\, f\,\right> \,\leq\, \left<\,S_{\Lambda}\,(\,f\,) \,,\, f\,\right> \,\leq\, \left<\,B\,f \,,\, f\,\right>\; \;\forall\; f \,\in\, H.\]
The operator \,$S_{\Lambda}$\; is bounded, self-adjoint, positive and invertible.\;Now, according to the Theorem (\ref{th1.001}), we can write, \,$A\,I_{\,H} \,\leq\,S_{\Lambda} \,\leq\, B\,I_{H}$\; and this gives \[ B^{\,-1}\,I_{H} \,\leq\, S_{\,\Lambda}^{\,-1} \,\leq\, A^{\,-1}\,I_{H}.\]
\end{definition}

\begin{definition}\cite{Sadri}
Let \,$\left\{\,W_{j}\,\right\}_{ j \,\in\, J}$\; be a collection of closed subspaces of \,$H$\; and \,$\left\{\,v_{j}\,\right\}_{ j \,\in\, J}$\; be a collection of positive weights and let \,$\Lambda_{j} \,\in\, \mathcal{B}\,(\,H \,,\, H_{j}\,)$\; for each \,$j \,\in\, J$\; and \,$K \,\in\, \mathcal{B}\,(\,H\,)$.\;Then the family \,$\Lambda \,=\, \left\{\,\left(\,W_{j},\, \Lambda_{j},\, v_{j}\,\right)\,\right\}_{j \,\in\, J}$\; is called a K-g-fusion frame for \,$H$\; if there exist constants \,$0 \,<\, A \,\leq\, B \,<\, \infty$\; such that
\begin{equation}\label{eq2}
A \,\left \|\,K^{\,\ast}\,f \,\right \|^{\,2} \,\leq\, \sum\limits_{\,j \,\in\, J}\,v_{j}^{\,2}\, \left\|\,\Lambda_{j}\,P_{\,W_{j}}\,(\,f\,) \,\right\|^{\,2} \,\leq\, B \, \left\|\, f \, \right\|^{\,2}\; \;\forall\; f \,\in\, H.
\end{equation} 
\end{definition}

\begin{theorem}\cite{Sadri}\label{th1.02}
Let \,$\Lambda$\; be a g-fusion Bessel sequence in \,$H$.\;Then \,$\Lambda$\; is a K-g-fusion frame for \,$H$\; if and only if there exists \,$A \,>\, 0$\; such that \,$S_{\Lambda} \,\geq\, A\,K\,K^{\,\ast}$. 
\end{theorem}

\begin{definition}\cite{Kutyniok}
A family of bounded operators \,$\left\{\,T_{j}\,\right\}_{j \,\in\, J}$\; on \,$H$\; is called a resolution of identity operator on \,$H$\; if for all \,$f \,\in\, H$, we have \,$f \,=\, \sum\limits_{\,j \,\in\, J}\,T_{j}\,(\,f\,)$, provided the series converges unconditionally for all \,$f \,\in\, H$.
\end{definition}

%=====================================
\section{Resolution of the identity operator in $g$-fusion frame}
%=====================================

\smallskip\hspace{.6 cm} In this section, we present several useful results of resolution of the identity operator on a Hilbert space using the theory of \,$g$-fusion frames. 

\begin{theorem}\label{th2}
Let \,$\Lambda \,=\, \left\{\,\left(\,W_{j},\, \Lambda_{j},\, v_{j}\,\right)\,\right\}_{j \,\in\, J}$\; be a g-fusion frame for \,$H$\, with frame bounds \,$C,\, D$\; and \,$S_{\Lambda}$\; be its associated g-fusion frame operator.\;Then the family \,$\left\{\,v_{j}^{\,2}\,P_{\,W_{j}}\,\Lambda_{j}^{\,\ast}\;T_{j}\,\right\}_{j \,\in\, J}$\; is resolution of the identity operator on \,$H$, where \,$T_{j} \,=\, \Lambda_{j}\,P_{\,W_{j}}\,S_{\,\Lambda}^{\,-1},\; j \,\in\, J$.\;Furthermore, for all \,$f \,\in\, H$, we have
\[\dfrac{C}{D^{\,2}}\; \|\,f\,\|^{\,2} \,\leq\, \sum\limits_{\,j \,\in\, J}\,v_{j}^{\,2}\, \|\,T_{j}\,(\,f\,)\,\|^{\,2} \,\leq\, \dfrac{D}{C^{\,2}}\; \|\,f\,\|^{\,2}.\] 
\end{theorem}

\begin{proof}
For any \,$f \,\in\, H$, we have the reconstruction formula for \,$g$-fusion frame:
\[f \,=\, S_{\Lambda}\,S_{\Lambda}^{\,-1}\,(\,f\,) \,=\, \sum\limits_{\,j \,\in\, J}\, v_{j}^{\,2}\; P_{\,W_{j}}\, \Lambda_{j}^{\,\ast}\; \Lambda_{j}\, P_{\,W_{j}}\, S_{\,\Lambda}^{\,-1}\,(\,f\,) \,=\, \sum\limits_{\,j \,\in\, J}\, v_{j}^{\,2}\; P_{\,W_{j}}\, \Lambda_{j}^{\,\ast}\; T_{j}\,(\,f\,).\] 
Thus, \,$\left\{\,v_{j}^{\,2}\; P_{\,W_{j}}\, \Lambda_{j}^{\,\ast}\; T_{j}\,\right\}_{j \,\in\, J}$\; is a resolution of the identity operator on \,$H$.\;Since \,$\Lambda$\; is a \,$g$-fusion frame with bounds \,$C$\, and \,$D$, for each \,$f \,\in\, H$, we have
\[\sum\limits_{\,j \,\in\, J}\, v_{j}^{\,2}\, \left\|\,T_{j}\,(\,f\,)\,\right\|^{\,2} \,=\, \sum\limits_{\,j \,\in\, J}\, v_{j}^{\,2}\, \left\|\,\Lambda_{j}\, P_{\,W_{j}}\, S_{\,\Lambda}^{\,-1}\,(\,f\,) \,\right\|^{\,2}\hspace{2cm}\]
\[\hspace{2.4cm} \,\leq\, D\, \left\|\,S_{\,\Lambda}^{\,-1}\,(\,f\,)\,\right\|^{\,2} \,\leq\, D\, \|\,S_{\,\Lambda}^{\,-1}\,\|^{\,2}\, \|\,f\,\|^{\,2}\]
\[\hspace{4cm} \,\leq\, \dfrac{D}{C^{\,2}}\, \|\,f\,\|^{\,2}\; [\;\text{since}\; \,D^{\,-1}\,I_{H} \,\leq\, S_{\,\Lambda}^{\,-1} \,\leq\, C^{\,-1}\,I_{H}\;].\] On the other hand,
\[\sum\limits_{\,j \,\in\, J}\, v_{j}^{\,2}\, \|\,T_{j}\,(\,f\,)\,\|^{\,2} \,=\, \sum\limits_{\,j \,\in\, J}\, v_{j}^{\,2}\, \left\|\,\Lambda_{j}\,P_{\,W_{j}}\,S_{\,\Lambda}^{\,-1}\,(\,f\,)\,\right\|^{\,2}\]
\[\,\geq\, C\, \left\|\,S_{\,\Lambda}^{\,-1}\,(\,f\,)\,\right\|^{\,2} \,\geq\, \dfrac{C}{D^{\,2}}\, \|\,f\,\|^{\,2}.\] Therefore,
\[\dfrac{C}{D^{\,2}}\; \|\,f\,\|^{\,2} \;\leq\; \sum\limits_{\,j \,\in\, J}\,v_{j}^{\,2}\, \|\,T_{j}\,(\,f\,)\,\|^{\,2} \;\leq\; \dfrac{D}{C^{\,2}}\; \|\,f\,\|^{\,2}\; \;\forall\, f \,\in\, H.\] 
\end{proof}

\begin{theorem}\label{th3}
Let \,$\Lambda \,=\, \left\{\,\left(\,W_{j},\, \Lambda_{j},\, v_{j}\,\right)\,\right\}_{j \,\in\, J}$\; be a g-fusion frame for \,$H$\; with frame bounds \,$C,\, D$\; and let \,$T_{j} \,:\, H \,\to\, H_{j}$\; be a bounded operator such that \,$\left\{\,v_{j}^{\,2}\;P_{\,W_{j}}\,\Lambda_{j}^{\,\ast}\;T_{j}\,\right\}_{j \,\in\, J}$\; is a resolution of the identity operator on \;$H$.\,Then
\[\dfrac{1}{D}\; \left\|\,\sum\limits_{\,j \,\in\, J}\, v_{j}^{\,2}\; P_{\,W_{j}}\, \Lambda_{j}^{\,\ast}\; T_{j}\,(\,f\,) \,\right\|^{\,2} \,\leq\, \sum\limits_{\,j \,\in\, J}\, v_{j}^{\,2}\; \left\|\,T_{j}\,(\,f\,)\,\right\|^{\,2}\; \;\forall\; f \,\in\, H.\]  
\end{theorem}

\begin{proof}
Assume \,$I \,\subset\, J$\, with \,$|\,I\,| \,<\, \infty$.\;If our inequality holds for all finite subsets then it would holds for all subset.\;Let \,$f \,\in\, H$\; and set \,$g \,=\, \sum\limits_{\,j \,\in\, I}\, v_{j}^{\,2}\, P_{\,W_{j}}\, \Lambda_{j}^{\,\ast}\, T_{j}\,(\,f\,)$.\;Then
\[\|\,g\,\|^{\,4} \;=\; (\,\left<\,g \,,\, g\,\right>\,)^{\,2} \;=\; \left(\,\left<\,g \;,\; \sum\limits_{\,j \,\in\, I}\,v_{j}^{\,2}\; P_{\,W_{j}}\, \Lambda_{j}^{\,\ast}\; T_{j}\,(\,f\,)\,\right>\,\right)^{\,2}\]
\[\;=\; \left(\,\sum\limits_{\,j \,\in\, I}\, v_{\,j}\; \left<\,\Lambda_{j}\, P_{\,W_{j}}\,(\,g\,) \;,\; v_{j}\, T_{j}\,(\,f\,) \,\right>\,\right)^{\,2}\hspace{.9cm}\]
\[\leq\; \left(\,\sum\limits_{\,j \,\in\, I}\, v_{j}\, \left\|\,\Lambda_{j}\, P_{\,W_{j}}\,(\,g\,) \,\right\|\; \left\|\,v_{j}\,T_{j}\,(\,f\,) \,\right\|\,\right)^{\,2}\hspace{.5cm}\]
\[\hspace{1.2cm}\leq\, \left(\,\sum\limits_{\,j \,\in\, I}\, v_{j}^{\,2}\,\left\|\,\Lambda_{j}\, P_{\,W_{j}}\,(\,g\,) \,\right\|^{\,2}\,\right)\,\left(\, \sum\limits_{\,j \,\in\, I}\, \left\|\,v_{j}\, T_{j}\,(\,f\,) \,\right\|^{\,2}\,\right)\]
\[\hspace{2.2cm}\leq\; D\; \|\,g\,\|^{\,2}\; \sum\limits_{\,j \,\in\, I}\,\left\|\, v_{j}\, T_{j}\,(\,f\,) \,\right\|^{\,2}\; [\;\text{since}\, \,\Lambda\; \text{is a $g$-fusion frame}\;].\]
\[\Rightarrow\, \dfrac{1}{D}\; \|\,g\,\|^{\,2} \;\leq\; \sum\limits_{\,j \,\in\, I}\, \left\|\,v_{j}\, T_{j}\,(\,f\,) \,\right\|^{\,2}\hspace{2cm}\]
\[\Rightarrow\, \dfrac{1}{D}\; \left\|\,\sum\limits_{\,j \,\in\, I}\,v_{j}^{\,2}\; P_{\,W_{j}}\, \Lambda_{j}^{\,\ast}\; T_{j}\,(\,f\,)\,\right\|^{\,2} \;\leq\; \sum\limits_{\,j \,\in\, I}\, v_{\,j}^{\,2}\; \left\|\,T_{\,j}\,(\,f\,)\,\right\|^{\,2}\; \;\forall\; f \,\in\, H.\]
Since the inequality holds for any finite subset \,$I \,\subset\, J$, we have
\[\dfrac{1}{D}\; \left\|\,\sum\limits_{\,j \,\in\, J}\, v_{j}^{\,2}\; P_{\,W_{j}}\, \Lambda_{j}^{\,\ast}\; T_{j}\,(\,f\,) \,\right\|^{\,2} \,\leq\, \sum\limits_{\,j \,\in\, J}\, v_{j}^{\,2}\; \left\|\,T_{j}\,(\,f\,)\,\right\|^{\,2}\; \;\forall\; f \,\in\, H.\]
This completes the proof.            
\end{proof}

\begin{theorem}
Let \,$\Lambda \,=\, \left\{\,\left(\,W_{j},\, \Lambda_{j},\, v_{j}\,\right)\,\right\}_{j \,\in\, J}$\; be a g-fusion frame for \,$H$\; with frame bounds \,$C,\, D$\; and let \,$T_{j} \,:\, H \,\to\, H_{j}$\; be a bounded operator such that \,$\left\{\,v_{j}^{\,2}\;P_{\,W_{j}}\,\Lambda_{j}^{\,\ast}\;T_{j}\,\right\}_{j \,\in\, J}$\; is a resolution of the identity operator on \,$H$.\;If \;$T^{\,\ast}_{j}\,\Lambda_{j}\,P_{\,W_{j}} \,=\, T_{j}$, then 
\[\dfrac{1}{D}\; \|\,f\,\|^{\,2} \,\leq\, \sum\limits_{\,j \,\in\, J}\,v_{j}^{\,2}\; \|\,T_{j}\,(\,f\,)\,\|^{\,2} \,\leq\, D\, E \;\|\,f\,\|^{\,2}\; \,\forall\, f \,\in\, H,\]where \,$E \,=\, \sup\limits_{\,j}\, \|\,T_{j}\,\|^{\,2} \,<\, \infty$. 
\end{theorem}

\begin{proof}
Since \,$\left\{\,v_{j}^{\,2}\;P_{\,W_{j}}\,\Lambda_{j}^{\,\ast}\;T_{j}\,\right\}_{j \,\in\, J}$\; is a resolution of the identity on \,$H$, 
\[ f \;=\; \sum\limits_{\,j \,\in\, J}\, v_{j}^{\,2}\; P_{\,W_{j}}\,\Lambda_{j}^{\,\ast}\; T_{j}\,(\,f\,),\; f \,\in\, H.\]
Now, for each \,$f \,\in\, H$, using Theorem (\ref{th3}), we get 
\[\dfrac{1}{D}\; \|\,f\,\|^{\,2}\; =\; \dfrac{1}{D}\;\left\|\,\sum\limits_{\,j \,\in\, J}\, v_{j}^{\,2}\; P_{\,W_{j}}\,\Lambda_{j}^{\,\ast}\; T_{j}\,(\,f\,)\,\right\|^{\,2} \;\leq\; \sum\limits_{\,j \,\in\, J}\, v_{j}^{\,2}\; \left\|\,T_{j}\,(\,f\,)\,\right\|^{\,2}\]
\[\hspace{1.6cm}=\; \sum\limits_{\,j \,\in\, J}\, v_{j}^{\,2}\; \left\|\,T^{\,\ast}_{j}\, \Lambda_{j}\, P_{\,W_{j}}\,(\,f\,)\,\right\|^{\,2}\; \;[\;\text{since}\; \,T^{\,\ast}_{j}\,\Lambda_{j}\,P_{\,W_{j}} \,=\, T_{j}\;]\]
\[\leq\; \sum\limits_{\,j \,\in\, J}\, v_{j}^{\,2}\; \|\,T_{j}\,\|^{\,2}\; \left\|\,\Lambda_{j}\, P_{\,W_{j}}\,(\,f\,)\,\right\|^{\,2}\hspace{2.1cm}\]
\[\hspace{1.45cm}\leq\; E\; \sum\limits_{\,j \,\in\, J}\, v_{j}^{\,2}\; \|\,\Lambda_{j}\, P_{\,W_{j}}\,(\,f\,)\,\|^{\,2}\; \;[\;\text{using}\; \,E \,=\, \sup\limits_{\,j}\, \|\,T_{\,j}\,\|^{\,2}\;]\]
\[ \;\leq\; D\, E \;\|\,f\,\|^{\,2}\; \;[\;\text{since}\; \,\Lambda\; \text{is a $g$-fusion frame}\;].\hspace{.5cm}\] This completes the proof.   
\end{proof}

\begin{theorem}
Let \,$\left\{\,W_{j}\,\right\}_{j \,\in\, J}$\, be a family of closed subspaces of \,$H$\, and \,$\left\{\,v_{j}\,\right\}_{j \,\in\, J}$\, be a family of bounded weights and let \,$\Lambda_{j} \,\in\, \mathcal{B}\,(\,H \,,\, H_{j}),\, j \,\in\, J$.\;Then \,$\Lambda \,=\, \left\{\,\left(\,W_{j},\, \Lambda_{j},\, v_{j}\,\right)\,\right\}_{j \,\in\, J}$\, is a g-fusion frame for \,$H$\, if the following conditions are hold:
\begin{itemize}
\item[(I)]\;For all \;$f \,\in\, H$, there exists \,$A \,>\, 0$\; such that
\[\sum\limits_{\,j \,\in\, J}\, \left\|\,\Lambda_{j}\, P_{\,W_{j}}\,(\,f\,) \,\right\|^{\,2} \;\leq\; \dfrac{1}{A} \; \left\|\, f \, \right\|^{\,2}.\]
\item[(II)]\;$\left\{\,v_{j}\, P_{\,W_{j}}\,\Lambda_{j}^{\,\ast}\; \Lambda_{j}\, P_{\,W_{j}}\,\right\}_{j \,\in\, J}$\; is a resolution of the identity operator on \,$H$.
\end{itemize}
\end{theorem}

\begin{proof}
Since \,$\left\{\,v_{j}\, P_{\,W_{j}}\, \Lambda_{j}^{\,\ast}\; \Lambda_{j}\, P_{\,W_{j}}\,\right\}_{j \,\in\, J}$\, is a resolution of the identity operator on \,$H$, for \,$f \,\in\, H$, we have
\[f \;=\; \sum\limits_{\,j \,\in\, J}\, v_{j}\, P_{\,W_{j}}\, \Lambda_{j}^{\,\ast}\; \Lambda_{j}\, P_{\,W_{j}}\,(\,f\,).\]By Cauchy-Schwarz inequality, we have 
\[\|\,f\,\|^{\,4} \,=\, \left(\,\left<\,f \,,\, f\,\right>\,\right)^{\,2} \,=\, \left(\,\left<\,\sum\limits_{\,j \,\in\, J}\, v_{j}\, P_{\,W_{j}}\, \Lambda_{j}^{\,\ast}\, \Lambda_{j}\, P_{\,W_{j}}\,(\,f\,) \,,\, f\,\right>\,\right)^{\,2}\]
\[\hspace{1.2cm} \,=\, \left(\,\sum\limits_{\,j \,\in\, J}\, v_{j}\,\left<\,\Lambda_{j}\, P_{\,W_{j}}\,(\,f\,) \;,\; \Lambda_{j}\,P_{\,W_{j}}\,(\,f\,)\,\right>\,\right)^{2}\,=\, \left(\,\sum\limits_{\,j \,\in\, J}\, v_{j}\, \left\|\,\Lambda_{j}\, P_{\,W_{j}}\,(\,f\,)\,\right\|^{\,2}\,\right)^{2}\]
\[ \leq\, \left(\,\sum\limits_{\,j \,\in\, J}\, \left\|\,\Lambda_{j}\, P_{\,W_{j}}\,(\,f\,)\,\right\|^{\,2}\,\right)\,  \left(\,\sum\limits_{\,j \,\in\, J}\, v_{j}^{\,2}\; \left\|\,\Lambda_{j}\, P_{\,W_{j}}\,(\,f\,)\,\right\|^{\,2}\,\right)\hspace{1.6cm}\]
\[\leq\; \dfrac{1}{A}\; \left\|\, f \, \right\|^{\,2}\; \sum\limits_{\,j \,\in\, J}\, v_{j}^{\,2}\; \left\|\,\Lambda_{j}\,P_{\,W_{j}}\,(\,f\,)\,\right\|^{\,2}\; \;[\;\text{using given condition $(\,I\,)$}\;]\]
\[\Rightarrow\, A\, \|\,f\,\|^{\,2} \,\leq\, \sum\limits_{j \,\in\, J}\, v_{j}^{\,2}\; \left\|\,\Lambda_{j}\, P_{\,W_{j}}\,(\,f\,)\,\right\|^{\,2}.\hspace{4.5cm}\] 
On the other hand,
\[\sum\limits_{j \,\in\, J}\, v_{j}^{\,2}\; \left\|\,\Lambda_{j}\, P_{\,W_{j}}\,(\,f\,)\,\right\|^{\,2} \;\leq\; B\, \sum\limits_{\,j \,\in\, J}\, \left\|\,\Lambda_{j}\, P_{\,W_{j}}\,(\,f\,) \,\right\|^{\,2}\; \;[\;\text{where}\; B \,=\, \sup\limits_{j \,\in\, J}\, \left\{\,v_{j}^{\,2}\,\right\}\;]\] 
\[\hspace{2.5cm}\leq\; \dfrac{\,B}{A}\; \|\,f\,\|^{\,2}\; \;[\;\text{using given condition $(\,I\,)$}\;]\] 
and hence \,$\Lambda$\; is a \,$g$-fusion frame.
\end{proof}

%=====================================
\section{$g$-atomic subspace}
%=====================================

\smallskip\hspace{.6 cm} In this section, we define generalized atomic subspace or \,$g$-atomic subspace of a Hilbert space with respect to a bounded linear operator.

\begin{definition}\label{def1}
Let \,$K \,\in\, \mathcal{B}\,(\,H\,)$\; and \,$\left\{\,W_{j}\,\right\}_{ j \,\in\, J}$\; be a collection of closed subspaces of \,$H$, let \,$\left\{\,v_{j}\,\right\}_{ j \,\in\, J}$\; be a collection of positive weights and \,$\Lambda_{j} \,\in\, \mathcal{B}\,(\,H \,,\, H_{j}\,)$\; for each \,$j \,\in\, J$.\;Then the family \;$\Lambda \,=\, \left\{\,\left(\,W_{j},\, \Lambda_{j},\, v_{j}\,\right)\,\right\}_{j \,\in\, J}$\; is said to be a generalized atomic subspace or g-atomic subspace of \,$H$\; with respect to \,$K$\; if the following statements hold:
\begin{itemize}
\item[(I)]\;$\Lambda$\; is a g-fusion Bessel sequence in \,$H$.
\item[(II)]\; For every \,$f \,\in\, H$, there exists \,$\left\{\,f_{\,j}\,\right\}_{j \,\in\, J} \,\in\, l^{\,2}\left(\,\left\{\,H_{j}\,\right\}_{ j \,\in\, J}\,\right)$\; such that 
\[K\,(\,f\,) \,=\, \sum\limits_{\,j \,\in\, J}\, v_{j}\, P_{\,W_{j}}\,\Lambda_{j}^{\,\ast}\; f_{\,j}\; \;\text{and}\; \;\left\|\,\left\{\,f_{\,j}\,\right\}_{j \,\in\, J}\,\right\|_{l^{\,2}\left(\,\left\{\,H_{j}\,\right\}_{ j \,\in\, J}\,\right)} \; \leq\; C\; \|\,f\,\|_{H}\]\; for some \,$C \,>\, 0$. 
\end{itemize}
\end{definition}

\begin{theorem}\label{thm1}
Let \,$K \,\in\, \mathcal{B}\,(\,H\,)$\; and \,$\left\{\,W_{j}\,\right\}_{ j \,\in\, J}$\; be a collection of closed subspaces of \,$H$, let \,$\left\{\,v_{j}\,\right\}_{ j \,\in\, J}$\; be a collection of positive weights and \,$\Lambda_{j} \,\in\, \mathcal{B}\,(\,H \,,\, H_{j}\,)$\; for each \,$j \,\in\, J$.\;Then the following statements are equivalent:
\begin{itemize}
\item[(I)]\; $\Lambda \,=\, \left\{\,\left(\,W_{j},\, \Lambda_{j},\, v_{j}\,\right)\,\right\}_{j \,\in\, J}$\; is a g-atomic subspace of \,$H$\, with respect to \,$K$.
\item[(II)]\; \,$\Lambda$\; is a K-g-fusion frame for \,$H$.
\end{itemize}
\end{theorem}

\begin{proof}$(I) \,\Rightarrow\, (II)$
Suppose \,$\Lambda$\, is a \,$g$-atomic subspace of \,$H$\; with respect to \,$K$.\;Then \,$\Lambda$\, is a \,$g$-fusion Bessel sequence, So, there exists \,$B \,>\, 0$\; such that
\[\sum\limits_{\,j \,\in\, J}\, v_{j}^{\,2}\; \left\|\,\Lambda_{j}\, P_{\,W_{j}}\,(\,f\,) \,\right\|^{\,2} \,\leq\, B\, \left\|\, f \, \right\|^{\,2}\; \;\forall\; f \,\in\, H.\] 
Now, for any \,$f \,\in\, H$, we have
\[\left \|\,K^{\,\ast}\,f \,\right \| \,=\, \sup\limits_{\,\|\,g\,\| \,=\, 1}\,\left|\,\left<\,K^{\,\ast}\,f \,,\, g\,\right>\,\right| \,=\, \sup\limits_{\,\|\,g\,\| \,=\, 1}\,\left|\,\left<\,f \,,\, K\,g\,\right>\,\right|,\]
by definition (\,\ref{def1}\,), for \,$g \,\in\, H$, there exists \,$\left\{\,f_{\,j}\,\right\}_{j \,\in\, J} \,\in\, l^{\,2}\left(\,\left\{\,H_{j}\,\right\}_{ j \,\in\, J}\,\right)$\; such that 
\[K\,(\,g\,) \,=\, \sum\limits_{\,j \,\in\, J}\, v_{j}\, P_{\,W_{j}}\,\Lambda_{j}^{\,\ast}\; f_{\,j}\; \;\text{and}\; \;\left\|\,\left\{\,f_{\,j}\,\right\}_{j \,\in\, J}\,\right\|_{l^{\,2}\left(\,\left\{\,H_{j}\,\right\}_{ j \,\in\, J}\,\right)} \; \leq\; C\; \|\,g\,\|_{H}\]\; for some \,$C \,>\, 0$.\;Thus
\[\left \|\,K^{\,\ast}\,f \,\right \| \,=\, \sup\limits_{\,\|\,g\,\| \,=\, 1}\,\left|\,\left<\,f \,,\, \sum\limits_{\,j \,\in\, J}\, v_{j}\, P_{\,W_{j}}\,\Lambda_{j}^{\,\ast}\,f_{\,j}\,\right>\,\right| \,=\, \sup\limits_{\,\|\,g\,\| \,=\, 1}\,\left|\,\sum\limits_{\,j \,\in\, J}\, v_{j}\,\left<\,\Lambda_{j}\,P_{\,W_{j}}\,(\,f\,) \,,\, \,f_{\,j}\,\right>\,\right|\]
\[\leq\, \sup\limits_{\,\|\,g\,\| \,=\, 1}\,\left(\,\sum\limits_{\,j \,\in\, J}\, v_{j}^{\,2}\; \left\|\,\Lambda_{j}\, P_{\,W_{j}}\,(\,f\,) \,\right\|^{\,2}\,\right)^{\dfrac{1}{2}}\,\left(\,\sum\limits_{\,j \,\in\, J}\,\left\|\,f_{\,j}\,\right\|^{\,2}\,\right)^{\dfrac{1}{2}}\]
\[\leq\, C\,\sup\limits_{\,\|\,g\,\| \,=\, 1}\,\left(\,\sum\limits_{\,j \,\in\, J}\, v_{j}^{\,2}\; \left\|\,\Lambda_{j}\, P_{\,W_{j}}\,(\,f\,) \,\right\|^{\,2}\,\right)^{\dfrac{1}{2}}\,\|\,g\,\|\hspace{1.7cm}\]
\[\Rightarrow\, \dfrac{1}{C^{\,2}}\,\left \|\,K^{\,\ast}\,f \,\right \|^{\,2} \,\leq\, \sum\limits_{\,j \,\in\, J}\, v_{j}^{\,2}\; \left\|\,\Lambda_{j}\, P_{\,W_{j}}\,(\,f\,) \,\right\|^{\,2}.\]
Therefore, \,$\Lambda$\; is a \,$K$-$g$-fusion frame for \,$H$\, with bounds \,$\dfrac{1}{C^{\,2}}$\, and \,$B$.\\\\
$(II) \,\Rightarrow\, (I)$\, Suppose that \,$\Lambda$\; is a \,$K$-$g$-fusion frame with the corresponding synthesis operator \,$T_{\Lambda}$.\;Then obviously \,$\Lambda$\; is a \,$g$-fusion Bessel sequence in \,$H$.\;Now, for each \,$f \,\in\, H$,
\[A\; \left \|\,K^{\,\ast}\,f \,\right \|^{\,2} \,\leq\, \sum\limits_{\,j \,\in\, J}\, v_{j}^{\,2}\; \left\|\,\Lambda_{j}\, P_{\,W_{j}}\,(\,f\,) \,\right\|^{\,2} \,=\, \|\,T_{\,\Lambda}^{\,\ast}\,f\,\|^{\,2}\]
gives \,$A \,K\,K^{\,\ast} \,\leq\, T_{\Lambda}\,T_{\Lambda}^{\,\ast}$\; and by Theorem (\ref{th1}), \,$\exists\; \;L \,\in\, \mathcal{B}\,\left(\,H \,,\, l^{\,2}\left(\,\left\{\,H_{j}\,\right\}_{ j \,\in\, J}\,\right)\,\right)$\, such that \,$K \,=\, T_{\Lambda}\,L$.\,Define \,$L\,(\,f\,) \,=\, \left\{\,f_{\,j}\,\right\}_{j \,\in\, J}$, for every \;$f \,\in\, H$.\;Then, for each \,$f \,\in\, H$, we have 
\[K \,(\,f\,) \,=\, T_{\Lambda}\,L\,(\,f\,) \,=\, T_{\,\Lambda}\,\left(\,\left\{\,f_{\,j}\,\right\}_{j \,\in\, J}\,\right) \,=\, \sum\limits_{\,j \,\in\, J}\, v_{j}\, P_{\,W_{j}}\, \Lambda_{j}^{\,\ast}\; f_{\,j}\, \;\text{and}\] 
\[\left\|\,\left\{\,f_{\,j}\,\right\}_{j \,\in\, J}\,\right\|_{l^{\,2}\left(\,\left\{\,H_{j}\,\right\}_{ j \,\in\, J}\,\right)} \,=\, \left\|\,L\,(\,f\,)\,\right\|_{l^{\,2}\left(\,\left\{\,H_{j}\,\right\}_{ j \,\in\, J}\,\right)} \,\leq\, C\; \|\,f\,\|,\]where \,$C \,=\, \|\,L\,\|$.\;Hence, \,$\Lambda$\; is a \,$g$-atomic subspace of \,$H$\, with respect to \,$K$.
\end{proof}

\begin{theorem}
Let \,$\Lambda \,=\, \left\{\,\left(\,W_{j},\, \Lambda_{j},\, v_{j}\,\right)\,\right\}_{j \,\in\, J}$\; be a g-fusion frame for \,$H$.\;Then \,$\Lambda$\; is a g-atomic subspace of \,$H$\; with respect to its g-fusion frame operator \,$S_{\Lambda}$.
\end{theorem}

\begin{proof}
Since \,$\Lambda$\; is a \,$g$-fusion frame in \,$H$, there exist \,$A,\, B \,>\, 0$\; such that
\[A\, \left\|\, f \, \right\|^{\,2} \,\leq\, \sum\limits_{\,j \,\in\, J}\, v_{j}^{\,2}\; \left\|\,\Lambda_{j}\, P_{\,W_{j}}\,(\,f\,) \,\right\|^{\,2} \,\leq\, B\, \left\|\, f \, \right\|^{\,2}\; \;\forall\; f \,\in\, H.\]Since \,$\mathcal{R}\,\left(\,T_{\Lambda}\,\right) \,=\, H \,=\, \mathcal{R}\,\left(\,S_{\Lambda}\,\right)$, by Theorem (\ref{th1}), there exists some \,$\alpha \,>\, 0$\; such that \,$\alpha \,S_{\Lambda}\, S_{\Lambda}^{\,\ast} \,\leq\, T_{\Lambda}\,T_{\Lambda}^{\,\ast}$\; and therefore for each \,$f \,\in\, H$, we have
\[\alpha\,\left \|\,S_{\Lambda}^{\,\ast}\,f \,\right \|^{\,2} \;\leq\; \|\,T_{\,\Lambda}^{\,\ast}\,f\,\|^{\,2} \,=\, \sum\limits_{\,j \,\in\, J}\, v_{j}^{\,2}\; \left\|\,\Lambda_{j}\, P_{\,W_{j}}\,(\,f\,) \,\right\|^{\,2} \,\leq\, B \; \left\|\, f \, \right\|^{\,2}.\] Thus, \,$\Lambda$\; is a \,$S_{\Lambda}$-$g$-fusion frame and hence by Theorem (\ref{thm1}), \,$\Lambda$\, is a \,$g$-atomic subspace of \,$H$\; with respect to \,$S_{\Lambda}$.
\end{proof}

\begin{theorem}\label{thm2}
Let \,$\Lambda \,=\, \left\{\,\left(\,W_{j},\, \Lambda_{j},\, v_{j}\,\right)\,\right\}_{j \,\in\, J}$\; and \,$\Gamma \,=\, \left\{\,\left(\,W_{j},\, \Gamma_{j},\, v_{j}\,\right)\,\right\}_{j \,\in\, J}$\; be two g-atomic subspaces of \,$H$\; with respect to \,$K \,\in\, \mathcal{B}\,(\,H\,)$\; with the corresponding synthesis operators \,$T_{\Lambda}$\; and \,$T_{\Gamma}$, respectively.\;If \;$T_{\Lambda}\, T_{\Gamma}^{\,\ast} \,=\, \theta_{H}\; (\,\theta_{H}$\, is a null operator on \,$H\,)$\, and \,$U,\, V \,\in\, \mathcal{B}\,(\,H\,)$\; such that \,$U \,+\, V$\; is invertible operator on \,$H$\, with \,$K\, \left(\,U \,+\, V \,\right) \,=\, \left(\,U \,+\, V \,\right)\, K$, then
\[\left\{\,\left(\,(\,U \,+\, V\,)\,W_{j},\; \left(\,\Lambda_{j} \,+\, \Gamma_{j}\,\right)\,P_{\,W_{j}}\,\left(\,U \,+\, V \,\right)^{\,\ast},\; v_{j}\,\right)\,\right\}_{j \,\in\, J}\] is a g-atomic subspace of \,$H$\; with respect to \,$K$.
\end{theorem}

\begin{proof}
Since \,$\Lambda$\; and \,$\Gamma$\; are \,$g$-atomic subspaces with respect to \,$K$, by Theorem (\ref{thm1}), they are \,$K$-$g$-fusion frames for \,$H$.\;So, for each \,$f \,\in\, H$, there exist positive constants \,$(\,A_{\,1},\, B_{\,1}\,)\, \;\text{and}\; \,(\,A_{\,2},\, B_{\,2}\,)$\; such that 
\[A_{\,1}\,\left \|\,K^{\,\ast}\,f \,\right \|^{\,2} \;\leq\; \sum\limits_{\,j \,\in\, J}\,v_{j}^{\,2}\, \left\|\,\Lambda_{j}\,P_{\,W_{j}}\,(\,f\,) \,\right\|^{\,2} \,\leq\, B_{\,1} \; \left\|\, f \, \right\|^{\,2}\, \;\text{and}\] 
\[A_{\,2}\,\left \|\,K^{\,\ast}\,f \,\right \|^{\,2} \;\leq\; \sum\limits_{\,j \,\in\, J}\,v_{j}^{\,2}\, \left\|\,\Gamma_{j}\,P_{\,W_{j}}\,(\,f\,) \,\right\|^{\,2} \,\leq\, B_{\,2} \; \left\|\, f \, \right\|^{\,2}.\]
Since \,$T_{\Lambda}\, T_{\Gamma}^{\,\ast} \,=\, \theta_{H}$, for any \,$f \,\in\, H$, we have
\begin{equation}\label{eqn3}
T_{\Lambda}\,\left\{\,v_{j}\,\Gamma_{j}\,P_{\,W_{j}}\,(\,f\,)\,\right\}_{j \,\in\, J} \,=\, \sum\limits_{\,j \,\in\, J}\,v_{j}^{\,2}\,P_{\,W_{j}}\,\Lambda_{j}^{\,\ast}\;\Gamma_{j}\,P_{\,W_{j}}\,(\,f\,)\,=\, 0\,.
\end{equation}
Also, \,$U \,+\, V$\; is invertible, so
\[\left\|\,K^{\,\ast}\,f\,\right\|^{\,2} \,=\, \left\|\,\left(\,\left(\,U \,+\, V\,\right)^{\,-\, 1}\,\right)^{\,\ast}\,(\,U \,+\, V\,)^{\,\ast}\,K^{\,\ast}\,f\,\right\|^{\,2}\]
\begin{equation}\label{eqn4}
\hspace{2cm}\leq\, \left\|\,\left(\,U \,+\, V\right)^{\,-\, 1}\,\right\|^{\,2}\, \left\|\,(\,U \,+\, V\,)^{\,\ast}\,K^{\,\ast}\,f\,\right\|^{\,2}.
\end{equation} 
Now, for any \,$f \,\in\, H$, we have  
\[\sum\limits_{\,j \,\in\, J}\,v_{j}^{\,2}\,\left\|\,\left(\,\Lambda_{j} \,+\, \Gamma_{j}\,\right)\,P_{\,W_{j}}\,\left(\,U \,+\, V \,\right)^{\,\ast}\, P_{\,(\,U \,+\, V\,)\,W_{j}}\,(\,f\,) \,\right\|^{\,2}\]
\[\,=\, \sum\limits_{\,j \,\in\, J}\,v_{j}^{\,2}\,\left\|\,\left(\,\Lambda_{j} \,+\, \Gamma_{j}\,\right)\,P_{\,W_{j}}\,\left(\,U \,+\, V\,\right)^{\,\ast}\,(\,f\,)\,\right\|^{\,2}\,[\;\text{using Theorem (\ref{th1.01})}\;]\hspace{2.7cm}\]
\[\,=\, \sum\limits_{\,j \,\in\, J}\,v_{j}^{\,2}\, \left<\,\left(\,\Lambda_{j} \,+\, \Gamma_{j}\,\right)\,P_{\,W_{j}}\,(\,T^{\,\ast}\,f\,) \;,\; \left(\,\Lambda_{j} \,+\, \Gamma_{j}\,\right)\,P_{\,W_{j}}\,(\,T^{\,\ast}\,f\,)\,\right>\; [\;\text{taking}\; \,T =\, U +\, V\,\;]\]
\[=\, \sum\limits_{\,j \,\in\, J}\,v_{j}^{\,2}\, \left(\, \left\|\,\Lambda_{j}\,P_{\,W_{j}}\,(\,T^{\,\ast}\,f\,)\,\right\|^{\,2} \,+\, \left\|\,\Gamma_{j}\,P_{\,W_{j}}\,(\,T^{\,\ast}\,f\,)\,\right\|^{\,2} \,+\, 2\,\text{Re}\, \left<\,T\,P_{\,W_{j}}\, \Lambda_{j}^{\,\ast}\, \Gamma_{j}\, P_{\,W_{j}}\,(\,T^{\,\ast}\,f\,) \,,\, f\,\right>\,\right)\]
\[=\; \sum\limits_{\,j \,\in\, J}\,v_{j}^{\,2}\, \left\|\,\Lambda_{j}\,P_{\,W_{j}}\,\left(\,T^{\,\ast}\,f\,\right)\,\right\|^{\,2} \,+\, \sum\limits_{\,j \,\in\, J}\,v_{j}^{\,2}\, \left\|\,\Gamma_{j}\,P_{\,W_{j}}\,\left(\,T^{\,\ast}\,f\,\right)\,\right\|^{\,2}\; \;[\;\text{using (\ref{eqn3})}\;]\hspace{2.3cm}\]
\[\leq\; B_{\,1}\; \left\|\,T^{\,\ast}\,f\,\right\|^{\,2} \,+\,  B_{\,2}\; \left\|\,T^{\,\ast}\,f\,\right\|^{\,2}\; \;[\;\text{since}\; \Lambda,\, \Gamma \;\text{are $K$-$g$-fusion frames}\;]\hspace{2.3cm}\]
\[=\, \left(\,B_{\,1} \,+\, B_{\,2}\,\right)\,\left\|\,\left(\,U \,+\, V\,\right)^{\,\ast}\,f\,\right\|^{\,2}\; \;[\;\text{since}\; \,T \,=\, U \,+\, V\;]\hspace{4.5cm}\] 
\[\leq\; \left(\,B_{\,1} \,+\, B_{\,2}\,\right)\, \left\|\,U \,+\, V\,\right\|^{\,2}\, \|\,f\,\|^{\,2}\; \;[\;\text{as}\; \,U \,+\, V \;\text{is bounded}\;].\hspace{3.6cm}\]
On the other hand, 
\[\sum\limits_{\,j \,\in\, J}\,v_{j}^{\,2}\,\left\|\,\left(\,\Lambda_{j} \,+\, \Gamma_{j}\,\right)\,P_{\,W_{j}}\,\left(\,U \,+\, V \,\right)^{\,\ast}\, P_{\,(\,U \,+\, V\,)\,W_{j}}\,(\,f\,) \,\right\|^{\,2}\hspace{2cm}\]
\[\hspace{2.3cm}=\; \sum\limits_{\,j \,\in\, J}\,v_{j}^{\,2}\, \left\|\,\Lambda_{j}\,P_{\,W_{j}}\,(\,U \,+\, V\,)^{\,\ast}\,f\,\right\|^{\,2} \,+\, \sum\limits_{\,j \,\in\, J}\,v_{j}^{\,2}\, \left\|\,\Gamma_{j}\,P_{\,W_{j}}\,(\,U \,+\, V\,)^{\,\ast}\,f\,\right\|^{\,2}\]
\[\geq\; \sum\limits_{\,j \,\in\, J}\,v_{j}^{\,2}\, \left\|\,\Lambda_{j}\,P_{\,W_{j}}\,(\,U \,+\, V\,)^{\,\ast}\,f\,\right\|^{\,2}\hspace{3.312cm}\]
\[\hspace{.7cm}\geq\; A_{\,1}\; \left\|\,K^{\,\ast}\,\left(\,U \,+\, V\,\right)^{\,\ast}\,f\,\right\|^{\,2}\;  \;[\;\text{since}\; \,\Lambda \;\text{is $K$-$g$-fusion frame}\;]\]
\[\hspace{1.9cm}=\; A_{\,1}\; \left\|\,\left(\,U \,+\, V\,\right)^{\,\ast}\,K^{\,\ast}\,f\,\right\|^{\,2}\;  \;[\;\text{using}\; \,K\, \left(\,U \,+\, V \,\right) \,=\, \left(\,U \,+\, V \,\right)\, K\;]\]
\[\hspace{2cm}\geq\; A_{\,1}\; \left\|\,\left(\,U \,+\, V\right)^{\,-\, 1}\,\right\|^{\,-\, 2}\; \left\|\,K^{\,\ast}\,f\,\right\|^{\,2}\;  \;[\;\text{using (\ref{eqn4})}\;].\hspace{3cm}\]
Therefore, \,$\left\{\,\left(\,(\,U \,+\, V)\,W_{j},\; \left(\,\Lambda_{j} \,+\, \Gamma_{j}\,\right)\,P_{\,W_{j}}\,\left(\,U \,+\, V \,\right)^{\,\ast},\; v_{j}\,\right)\,\right\}_{j \,\in\, J}$\; is a \,$K$-$g$-fusion frame and by Theorem (\ref{thm1}), it is a \,$g$-atomic subspace of \,$H$\, with respect to \,$K$.
\end{proof}

\begin{corollary}
Let \,$\Lambda \,=\, \left\{\,\left(\,W_{j},\, \Lambda_{j},\, v_{j}\,\right)\,\right\}_{j \,\in\, J}$\, and \;$\Gamma \,=\, \left\{\,\left(\,W_{j},\, \Gamma_{j},\, v_{j}\,\right)\,\right\}_{j \,\in\, J}$\; be two g-atomic subspaces of \,$H$\; with respect to \,$K \,\in\, \mathcal{B}\,(\,H\,)$\; with the corresponding synthesis operators \,$T_{\Lambda}$\; and \,$T_{\Gamma}$.\;If \,$T_{\Lambda}\, T_{\Gamma}^{\,\ast} \,=\, \theta_{H}$\; and \,$U \,\in\, \mathcal{B}\,(\,H\,)$\; is an invertible operator with \,$K\,U \,=\, U\,K$\,, then \;$\left\{\,\left(\,U\,W_{j},\; \left(\,\Lambda_{j} \,+\, \Gamma_{j}\,\right)\,P_{\,W_{j}}\,U^{\,\ast},\; v_{j}\,\right)\,\right\}_{j \,\in\, J}$\; is a g-atomic subspace of \,$H$\; with respect to \,$K$. 
\end{corollary}

\begin{proof}
The proof of this Corollary directly follows from the Theorem (\ref{thm2}), by putting \,$V \,=\, \theta_{H}$.
\end{proof}

\begin{theorem}
Let \,$\Lambda \,=\, \left\{\,\left(\,W_{j},\, \Lambda_{j},\, v_{j}\,\right)\,\right\}_{j \,\in\, J}$\, be a g-atomic subspace for \,$K \,\in\, \mathcal{B}\,(\,H\,)$\, and \,$S_{\Lambda}$\, be the frame operator of \,$\Lambda$.\;If \,$U \,\in\, \mathcal{B}\,(\,H\,)$\, be a positive and invertible operator on \,$H$, then \,$\Lambda^{\,\prime} \,=\, \left\{\,\left(\,(\,I_{H} \,+\, U\,)\,W_{j},\; \Lambda_{j}\,P_{\,W_{j}} \left(\,I_{H} \,+\, U \,\right)^{\,\ast},\; v_{j}\,\right)\,\right\}_{j \,\in\, J}$\, is a g-atomic subspace of \,$H$\, with respect to \,$K$.\;Moreover, for any natural number \,$n$, \,$\Lambda^{\,\prime\,\prime} \,=\, \left\{\,\left(\,(\,I_{H} \,+\, U^{\,n}\,)\,W_{j},\; \Lambda_{j}\,P_{\,W_{j}} \left(\,I_{H} \,+\, U^{\,n} \,\right)^{\,\ast},\; v_{j}\,\right)\,\right\}_{j \,\in\, J}$\; is a g-atomic subspace of \,$H$\; with respect to \,$K$.  
\end{theorem}

\begin{proof}
Since \,$\Lambda$\; is a \,$g$-atomic subspace with respect to \,$K$, by Theorem (\ref{thm1}), it is a \,$K$-$g$-fusion frame for \,$H$.\;Then, according to the Theorem (\ref{th1.02}), there exists \,$A \,>\, 0$\; such that \,$S_{\Lambda} \,\geq\, A\,K\,K^{\,\ast}$.\;Now, for each \,$f \,\in\, H$, we have
\[\sum\limits_{\,j \,\in\, J}\,v_{j}^{\,2}\,\left\|\,\Lambda_{j}\,P_{\,W_{j}} \left(\,I_{H} \,+\, U \,\right)^{\,\ast}\, P_{\,\left(\,I_{H} \,+\, U\,\right)\,W_{j}}\,(\,f\,)\,\right\|^{\,2}\]
\[\hspace{.7cm}\,=\, \sum\limits_{\,j \,\in\, J}\,v_{j}^{\,2}\,\left\|\,\Lambda_{j}\,P_{\,W_{j}} \left(\,I_{H} \,+\, U \,\right)^{\,\ast}\, (\,f\,)\,\right\|^{\,2}\; \;[\;\text{using Theorem (\ref{th1.01})}\;]\]
\[\leq\, B\,\left\|\,\left(\,I_{H} \,+\, U \,\right)^{\,\ast}\, (\,f\,)\,\right\|^{\,2}\; \;[\;\text{since $\Lambda$ is a $K$-$g$-fusion frame}\;]\]
\[\leq\, B\, \left\|\,I_{H} \,+\, U\,\right\|^{\,2}\,\|\,f\,\|^{\,2}\; \;[\;\text{since \,$\left(\,I_{H} \,+\, U\,\right) \,\in\, \mathcal{B}\,(\,H\,)$}\;].\hspace{.7cm}\]
Thus \,$\Lambda^{\,\prime}$\, is a \,$g$-fusion Bessel sequence in \,$H$.\;Also, for each \,$f \,\in\, H$, we have 
\[\sum\limits_{\,j \,\in\, J}\,v_{j}^{\,2}\; P_{\,\left(\,I_{H} \,+\, U\,\right)\,W_{j}}\, \left(\,\Lambda_{j}\,P_{\,W_{j}} \left(\,I_{H} \,+\, U \,\right)^{\,\ast}\,\right)^{\,\ast}\, \Lambda_{j}\,P_{\,W_{j}} \left(\,I_{H} \,+\, U \,\right)^{\,\ast}\, P_{\,\left(\,I_{H} \,+\, U\,\right)\,W_{j}}\,(\,f\,)\]
\[=\, \sum\limits_{\,j \,\in\, J}\,v_{j}^{\,2}\, P_{\,\left(\,I_{H} \,+\, U\,\right)\,W_{j}} \left(\,I_{H} \,+\, U \,\right)\,P_{\,W_{j}}\, \Lambda^{\,\ast}_{j}\,\Lambda_{j}\,P_{\,W_{j}} \left(\,I_{H} \,+\, U \,\right)^{\,\ast}\, P_{\,\left(\,I_{H} \,+\, U\,\right)\,W_{j}}\,(\,f\,)\]
\[=\, \sum\limits_{\,j \,\in\, J}\,v_{j}^{\,2}\,\left(\,P_{\,W_{j}}\left(\,I_{H} \,+\, U \,\right)^{\,\ast}\, P_{\,\left(\,I_{H} \,+\, U\,\right)\,W_{j}}\,\right)^{\,\ast}\, \Lambda^{\,\ast}_{j}\,\Lambda_{j}\,\left(\,P_{\,W_{j}} \left(\,I_{H} \,+\, U \,\right)^{\,\ast}\, P_{\,\left(\,I_{H} \,+\, U\,\right)\,W_{j}}\,(\,f\,)\,\right)\]
\[=\, \sum\limits_{\,j \,\in\, J}\,v_{j}^{\,2}\, \left(\,P_{\,W_{j}}\,\left(\,I_{H} \,+\, U\,\right)^{\,\ast}\,\right)^{\,\ast}\, \Lambda^{\,\ast}_{j}\,\Lambda_{j}\, P_{\,W_{j}} \left(\,I_{H} \,+\, U\,\right)^{\,\ast}\,(\,f\,)\; \;[\;\text{using Theorem (\ref{th1.01})}\;]\]
\[=\, \sum\limits_{\,j \,\in\, J}\,v_{j}^{\,2} \left(\,I_{H} \,+\, U\,\right)\, P_{\,W_{j}}\, \Lambda^{\,\ast}_{j}\,\Lambda_{j}\, P_{\,W_{j}}\left(\,I_{H} \,+\, U\,\right)^{\,\ast}\,(\,f\,) \hspace{5cm}\]
\[ =\, \left(\,I_{H} \,+\, U\,\right)\, \sum\limits_{\,j \,\in\, J}\,v_{j}^{\,2}\,P_{\,W_{j}}\, \Lambda^{\,\ast}_{j}\,\Lambda_{j}\, P_{\,W_{j}} \left(\,I_{H} \,+\, U\,\right)^{\,\ast}\,(\,f\,) \,=\, \left(\,I_{H} \,+\, U\,\right)\,S_{\Lambda}\, \left(\,I_{H} \,+\, U\,\right)^{\,\ast}\,(\,f\,).\]This shows that the frame operator of \,$\Lambda^{\,\prime}$\; is \,$\left(\,I_{H} \,+\, U\,\right)\,S_{\Lambda}\, \left(\,I_{H} \,+\, U\,\right)^{\,\ast}$.\;Now, 
\[ \left(\,I_{H} \,+\, U\,\right)\,S_{\Lambda}\, \left(\,I_{H} \,+\, U\,\right)^{\,\ast} \,\geq\, S_{\Lambda}\, \geq\, A\,K\,K^{\,\ast}\; \;[\;\text{since}\; \,U,\, S_{\Lambda}\; \;\text{are positive}\;].\] Then by Theorem (\ref{th1.02}), we can conclude that \,$\Lambda^{\,\prime}$\; is a \,$K$-$g$-fusion frame and therefore by Theorem (\ref{thm1}), \,$\Lambda^{\,\prime}$\; is a \,$g$-atomic subspace of \,$H$\, with respect to \,$K$.
According to the preceding procedure, for any natural number \,$n$, the frame operator of \,$\Lambda^{\,\prime\,\prime}$\; is \,$\left(\,I_{H} \,+\, U^{\,n}\,\right)\,S_{\Lambda}\, \left(\,I_{H} \,+\, U^{\,n}\,\right)^{\,\ast}$\; and similarly it can be shown that \,$\Lambda^{\,\prime\,\prime}$\; is a g-atomic subspace of \,$H$\; with respect to \,$K$.        
\end{proof}

%=====================================
\section{Frame operator for a pair of $g$-fusion Bessel sequences}
%=====================================

\smallskip\hspace{.6 cm} In this section, we shall discussed about the frame operator for a pair of \,$g$-fusion Bessel sequences and established some properties relative to frame operator.\;At the end of this section, we shall construct a new \,$g$-fusion frame for the Hilbert space \,$H \,\oplus\, X$, using the \,$g$-fusion frames of the Hilbert spaces \,$H$\, and \,$X$.

\begin{definition}
Let \,$\Lambda \,=\, \left\{\,\left(\,W_{j},\, \Lambda_{j},\, w_{j}\,\right)\,\right\}_{j \,\in\, J}$\, and \,$\Gamma \,=\, \left\{\,\left(\,V_{j},\, \Gamma_{j},\, v_{j}\,\right)\,\right\}_{j \,\in\, J}$\, be two g-fusion Bessel sequences in \,$H$\; with bounds \,$D_{\,1}$\, and \,$D_{\,2}$.\;Then the operator \,$S_{\Gamma\,\Lambda} \,:\, H \,\to\, H$, defined by
\[S_{\Gamma\,\Lambda}\,(\,f\,) \,=\, \sum\limits_{\,j \,\in\, J}\,v_{j}\, w_{j}\, P_{\,V_{j}}\, \Gamma_{j}^{\,\ast}\; \Lambda_{j}\, P_{\,W_{j}}\,(\,f\,)\;\; \;\forall\; f \,\in\, H,\] is called the frame operator for the pair of g-fusion Bessel sequences \,$\Lambda$\, and \,$\Gamma$.
\end{definition}

\begin{theorem}
The frame operator \,$S_{\Gamma\,\Lambda}$\; for the pair of g-fusion Bessel sequences \,$\Lambda$\; and \,$\Gamma$\; is bounded and \,$S_{\,\Gamma\,\Lambda}^{\,\ast} \,=\, S_{\Lambda\,\Gamma}$.  
\end{theorem}

\begin{proof}
For each \,$f,\; g \,\in\, H$, we have
\[\left<\,S_{\Gamma\,\Lambda}\,(\,f\,) \,,\, g\,\right> \,=\, \left<\,\sum\limits_{\,j \,\in\, J}\, v_{j}\, w_{j}\, P_{\,V_{j}}\, \Gamma_{j}^{\,\ast}\; \Lambda_{j}\, P_{\,W_{j}}\,(\,f\,) \;,\; g\,\right>\]
\begin{equation}\label{eq3}
\hspace{2.8cm} =\; \sum\limits_{\,j \,\in\, J}\, v_{j}\, w_{j}\,\left<\,\Lambda_{j}\, P_{\,W_{j}}\,(\,f\,) \;,\; \Gamma_{j}\,P_{\,V_{j}}\,(\,g\,)\,\right>.
\end{equation} 
By the Cauchy-Schwarz inequality, we obtain \,$\left|\,\left<\,S_{\Gamma\,\Lambda}\,(\,f\,) \,,\, g\,\right>\,\right|$
\begin{equation}\label{eq4}
\hspace{1cm} \,\leq\, \left(\,\sum\limits_{\,j \,\in\, J}\, v_{\,j}^{\,2}\; \left\|\,\Gamma_{j}\, P_{\,V_{j}}\,(\,g\,) \,\right\|^{\,2}\,\right)^{\,\dfrac{1}{2}}\; \left(\,\sum\limits_{\,j \,\in\, J}\, w_{\,j}^{\,2}\;  \left\|\,\Lambda_{j}\, P_{\,W_{j}}\,(\,f\,) \,\right\|^{\,2}\,\right)^{\,\dfrac{1}{2}}
\end{equation}
\[\leq\; \sqrt{D_{\,2}}\; \;\|\,g\,\|\; \sqrt{D_{\,1}}\; \;\|\,f\,\|.\hspace{7.4cm}\] This shows that \,$S_{\Gamma\,\Lambda}$\; is a bounded operator with \,$\|\,S_{\Gamma\,\Lambda}\,\| \;\leq\; \sqrt{D_{\,1}\, D_{\,2}}$.\;Now, 
\[\left\|\,S_{\Gamma\,\Lambda}\,f\,\right\| \;=\; \sup\left\{\,\left|\,\left<\,S_{\Gamma\,\Lambda}\,(\,f\,) \,,\, g\,\right>\,\right| \;:\; \|\,g\,\| \;=\; 1\,\right\}\]
\[\leq\, \sup\left\{\,\sqrt{D_{\,2}}\; \;\|\,g\,\|\; \left(\,\sum\limits_{\,j \,\in\, J}\, w_{\,j}^{\,2}\; \left\|\,\Lambda_{j}\,P_{\,W_{j}}\,(\,f\,) \,\right\|^{\,2}\,\right)^{\,\dfrac{1}{2}} \;:\; \|\,g\,\| \,=\, 1\,\right\}\; \;[ \;\text{using}\; (\ref{eq4})\;]\]
\begin{equation}\label{eq5}
\;\leq\; \sqrt{D_{\,2}}\; \left(\,\sum\limits_{\,j \,\in\, J}\, w_{\,j}^{\,2}\; \left\|\,\Lambda_{j}\,P_{\,W_{j}}\,(\,f\,) \,\right\|^{\,2}\,\right)^{\,\dfrac{1}{2}}  
\end{equation}
and similarly it can be shown that
\begin{equation}\label{eq6}
\left\|\,S_{\Gamma\,\Lambda}^{\,\ast}\,g\,\right\| \;\leq\; \sqrt{D_{\,1}}\; \left(\,\sum\limits_{\,j \,\in\, J}\, v_{\,j}^{\,2}\; \left\|\,\Gamma_{j}\,P_{\,V_{j}}\,(\,g\,) \,\right\|^{\,2}\,\right)^{\,\dfrac{1}{2}}. 
\end{equation}
Also, for each \,$f,\, g \,\in\, H$, we have
\[\left<\,S_{\Gamma\,\Lambda}\,(\,f\,) \;,\; g\,\right> \,=\, \left<\,\sum\limits_{\,j \,\in\, J}\,v_{j}\, w_{j}\, P_{\,V_{j}}\, \Gamma_{j}^{\,\ast}\; \Lambda_{j}\, P_{\,W_{j}}\,(\,f\,) \;,\; g\,\right>\hspace{2.5cm}\]
\[=\, \sum\limits_{\,j \,\in\, J}\, v_{\,j}\, w_{\,j}\, \left<\,f \;,\;  P_{\,W_{j}}\, \Lambda_{j}^{\,\ast}\; \Gamma_{j}\, P_{\,V_{j}}\,(\,g\,) \,\right>\]
\[\hspace{3.2cm}=\; \left<\,f \,,\, \sum\limits_{\,j \,\in\, J}\,w_{\,j}\,v_{\,j}\,P_{\,W_{j}}\,\Lambda_{j}^{\,\ast}\,\Gamma_{j}\,P_{\,V_{j}}\,(\,g\,)\,\right> \,=\, \left<\,f \;,\; S_{\Lambda\,\Gamma}\,(\,g\,)\,\right>\]and hence \,$S_{\Gamma\,\Lambda}^{\,\ast} \,=\, S_{\Lambda\,\Gamma}$. 
\end{proof}

\begin{theorem}
Let \,$S_{\Gamma\,\Lambda}$\; be the frame operator for a pair of g-fusion Bessel sequences \,$\Lambda$\, and \,$\Gamma$\, with bounds \,$D_{\,1}$\, and \,$D_{\,2}$, respectively.\;Then the following statements are equivalent:
\begin{itemize}
\item[(I)]\; $S_{\Gamma\,\Lambda}$\; is bounded below.
\item[(II)] There exists \,$K \,\in\, \mathcal{B}\,(\,H\,)$\; such that \;$\left\{\,T_{j}\,\right\}_{j \,\in\, J}$\; is a resolution of the identity operator on \,$H$, where \;$T_{j} \,=\, v_{j}\, w_{j}\, K\, P_{\,V_{j}}\, \Gamma_{j}^{\,\ast}\; \Lambda_{j}\, P_{\,W_{j}} \;,\; j \,\in\, J$.  
\end{itemize} 
If one of the given conditions hold, then \,$\Lambda$\; is a g-fusion frame.
\end{theorem}

\begin{proof}
$(\,I\,) \,\Rightarrow\, (\,II\,)$\; Suppose that \,$S_{\,\Gamma\,\Lambda}$\, is bounded below.\;Then for each \,$f \,\in\, H$, there exists \,$A \,>\, 0$\; such that  
\[\|\,f\,\|^{\,2} \,\leq\, A\;\left\|\,S_{\,\Gamma\,\Lambda}\,f\,\right\|^{\,2} \,\Rightarrow\, \left<\,I_{H}\,f \,,\, f\,\right> \,\leq\, A\, \left<\,S_{\,\Gamma\,\Lambda}^{\,\ast}\, S_{\,\Gamma\,\Lambda}\,f \,,\, f\,\right> \,\Rightarrow\,I^{\,\ast}_{H}\,I_{\,H} \,\leq\, A\, S_{\,\Gamma\,\Lambda}^{\,\ast}\, S_{\,\Gamma\,\Lambda}.\]So, by Theorem (\ref{th1}), there exists \,$K \,\in\, \mathcal{B}\,(\,H\,)$\, such that \,$K\, S_{\,\Gamma\,\Lambda} \,=\, I_{\,H}$.\;Therefore, for each \,$f \,\in\, H$, we have 
\[ f \,=\, K\, S_{\,\Gamma\,\Lambda}\,(\,f\,) \,=\, K\, \left(\,\sum\limits_{\,j \,\in\, J}\, v_{j}\, w_{j}\, P_{\,V_{j}}\, \Gamma_{j}^{\,\ast}\; \Lambda_{j}\, P_{\,W_{j}}\,(\,f\,)\,\right)\]
\[\,=\, \sum\limits_{\,j \,\in\, J}\, v_{j}\, w_{j}\, K\, P_{\,V_{j}}\, \Gamma_{j}^{\,\ast}\, \Lambda_{j}\, P_{\,W_{j}}\,(\,f\,) \,=\, \sum\limits_{\,j \,\in\, J}\, T_{j}\,(\,f\,)\hspace{.6cm}\] and hence \,$\left\{\,T_{\,j}\,\right\}_{j \,\in\, J}$\; is a resolution of the identity operator on \,$H$, where \,$T_{j} \,=\, v_{j}\, w_{j}\, K\, P_{\,V_{j}}\, \Gamma_{j}^{\,\ast}\, \Lambda_{j}\, P_{\,W_{j}}$.\\\\
$(\,II\,) \,\Rightarrow\, (\,I\,)$\; Since \;$\left\{\,T_{j}\,\right\}_{j \,\in\, J}$\; is a resolution of the identity operator on \,$H$, for any \,$f \,\in\, H$, we have
\[f \,=\, \sum\limits_{\,j \,\in\, J}\, T_{\,j}\,(\,f\,) \,=\, \sum\limits_{\,j \,\in\, J}\, v_{j}\, w_{j}\, K\, P_{\,V_{j}} \,\Gamma_{j}^{\,\ast}\; \Lambda_{j}\, P_{\,W_{j}}\,(\,f\,)\]
\[\hspace{1.3cm} \,=\, K\, \left(\,\sum\limits_{\,j \,\in\, J}\, v_{j}\, w_{j}\, P_{\,V_{j}}\, \Gamma_{j}^{\,\ast}\; \Lambda_{j}\, P_{\,W_{j}}\,(\,f\,)\,\right) \,=\, K\, S_{\,\Gamma\,\Lambda}\,(\,f\,).\]Thus, \,$I_{H} \,=\, K\,S_{\,\Gamma\,\Lambda}$.\;So, by Theorem (\ref{th1}), there exists some \,$\alpha \,>\, 0$\; such that \,$I_{H}\, I_{H}^{\,\ast} \,\leq\; \alpha\; S_{\,\Gamma\,\Lambda}\, S_{\,\Gamma\,\Lambda}^{\,\ast}$\; and hence \;$S_{\,\Gamma\,\Lambda}$\; is bounded below.\\

Last part:
First we suppose that \,$S_{\,\Gamma\,\Lambda}$\; is bounded below.\;Then for all \\$f \,\in\, H$, there exists \,$M \,>\, 0$\; such that \,$\left\|\,S_{\,\Gamma\,\Lambda}\,f\,\right\| \,\geq\, M\; \|\,f\,\|$\, and this implies that 
\[M^{\,2}\; \|\,f\,\|^{\,2} \;\leq\; \left\|\,S_{\,\Gamma\,\Lambda}\,f\,\right\|^{\,2} \;\leq\; D_{\,2}\; \sum\limits_{\,j \,\in\, J}\,w_{j}^{\,2}\, \left\|\,\Lambda_{j}\, P_{\,W_{j}}\,(\,f\,) \,\right\|^{\,2}\; \;[\;\text{using (\ref{eq5})}\;]\]
\[\Rightarrow\; \dfrac{M^{\,2}}{D_{\,2}}\; \|\,f\,\|^{\,2} \;\leq\; \sum\limits_{\,j \,\in\, J}\, w_{j}^{\,2}\; \left\|\,\Lambda_{j}\, P_{\,W_{j}}\,(\,f\,) \,\right\|^{\,2}.\hspace{1cm}\] Hence, \,$\Lambda$\; is a \,$g$-fusion frame for \,$H$\; with bounds \,$ \dfrac{M^{\,2}}{D_{\,2}}$\, and \,$D_{\,1}$.\\Next, we suppose that the given condition $(II)$ holds.\;Then for any \,$f \,\in\, H$, we have
\[f \,=\, \sum\limits_{\,j \,\in\, J}\, v_{j}\, w_{j}\, K\, P_{\,V_{j}} \,\Gamma_{j}^{\,\ast}\; \Lambda_{j}\, P_{\,W_{j}}\,(\,f\,),\; K \,\in\, \mathcal{B}\,(\,H\,).\]By Cauchy-Schwarz inequality, for each \,$f  \,\in\, H$, we have
\[\|\,f\,\|^{\,2} \,=\, \left<\,f \,,\, f\,\right> \,=\, \left<\,\sum\limits_{\,j \,\in\, J}\, v_{j}\, w_{j}\, K\, P_{\,V_{j}} \,\Gamma_{j}^{\,\ast}\; \Lambda_{j}\, P_{\,W_{j}}\,(\,f\,) \;,\; f\,\right>\] 
\[=\, \sum\limits_{\,j \,\in\, J}\, v_{j}\, w_{j}\,\left<\,\Lambda_{j}\, P_{\,W_{j}}\,(\,f\,) \;,\; \Gamma_{j}\,P_{\,V_{j}}\,(\,K^{\,\ast}\,f\,)\,\right>\hspace{.3cm}\]
\[\hspace{2.95cm}\;\leq\,  \left(\,\sum\limits_{\,j \,\in\, J}\, w_{\,j}^{\,2}\,  \left\|\,\Lambda_{j}\, P_{\,W_{j}}\,(\,f\,) \,\right\|^{\,2}\,\right)^{\,\dfrac{1}{2}}\, \left(\,\sum\limits_{\,j \,\in\, J}\, v_{\,j}^{\,2}\, \left\|\,\Gamma_{j}\, P_{\,V_{j}}\,(\,K^{\,\ast}\,f\,) \,\right\|^{\,2}\,\right)^{\,\dfrac{1}{2}}\]
\[\hspace{.4cm}\leq\, \sqrt{D_{\,2}}\, \left\|\,K^{\,\ast}\,f\right\|\, \left(\,\sum\limits_{\,j \,\in\, J}\, w_{\,j}^{\,2}\; \left\|\,\Lambda_{j}\, P_{\,W_{j}}\,(\,f\,) \,\right\|^{\,2}\,\right)^{\,\dfrac{1}{2}}\]  
\[\hspace{.7cm}\leq\, \sqrt{D_{\,2}}\, \|\,K\,\|\, \|\,f\,\|\, \left(\,\sum\limits_{\,j \,\in\, J}\, w_{\,j}^{\,2}\; \left\|\,\Lambda_{j}\, P_{\,W_{j}}\,(\,f\,) \,\right\|^{\,2}\,\right)^{\,\dfrac{1}{2}}\]
\[\hspace{.3cm}\Rightarrow\, \dfrac{1}{D_{\,2}\,\|\,K\,\|^{\,2}}\, \|\,f\,\|^{\,2} \,\leq\, \sum\limits_{\,j \,\in\, J}\, w_{\,j}^{\,2}\; \left\|\,\Lambda_{j}\, P_{\,W_{j}}\,(\,f\,) \,\right\|^{\,2}.\]Therefore, in this case \,$\Lambda$\; is also a \,$g$-fusion frame for \,$H$.       
\end{proof}

\begin{theorem}
Let \,$S_{\Gamma\,\Lambda}$\, be the frame operator for a pair of g-fusion Bessel sequences \,$\Lambda$\, and \,$\Gamma$\, with bounds \,$D_{\,1}$\, and \,$D_{\,2}$, respectively.\;Suppose \;$\lambda_{\,1} \,<\, 1,\; \lambda_{\,2} \,>\, \,-\,1$\; such that for each \,$f \,\in\, H$, 
\[\left\|\,f \,-\, S_{\Gamma\,\Lambda}\,f\,\right\| \,\leq\; \lambda_{\,1}\, \|\,f\,\| \,+\, \lambda_{\,2}\, \left\|\,S_{\Gamma\,\Lambda}\,f \,\right\|.\] Then \,$\Lambda$\, is a g-fusion frame for \;$H$.   
\end{theorem}

\begin{proof}
For each \,$f \,\in\, H$, we have
\[ \|\,f\,\| \,-\, \left\|\,S_{\,\Gamma\,\Lambda}\,f \,\right\|\, \leq\; \left\|\,f \,-\, S_{\,\Gamma\,\Lambda}\,f\,\right\| \,\leq\, \lambda_{\,1}\, \|\,f\,\| \,+\, \lambda_{\,2}\, \left\|\,S_{\,\Gamma\,\Lambda}\,f \,\right\|\]
\[\Rightarrow\; \left(\,1 \,-\, \lambda_{\,1}\,\right)\,\|\,f\,\| \;\leq\; \left(\,1 \,+\, \lambda_{\,2}\,\right)\;\|\,S_{\,\Gamma\,\Lambda}\,f \,\|\hspace{3.7cm}\] 
\[\hspace{1.5cm}\Rightarrow\, \left(\,\dfrac{1 \,-\, \lambda_{\,1}}{1 \,+\, \lambda_{\,2}}\,\right)\; \|\,f\,\| \;\leq\; \sqrt{D_{\,2}}\; \left(\,\sum\limits_{\,j \,\in\, J}\,w_{j}^{\,2}\, \left\|\,\Lambda_{j}\,P_{\,W_{j}}\,(\,f\,) \,\right\|^{\,2}\,\right)^{\,\dfrac{1}{2}}\; \;[\;\text{using (\ref{eq5})}\;]\]   
\begin{equation}\label{eq7}
\Rightarrow\; \dfrac{1}{D_{\,2}}\;\left(\,\dfrac{1 \,-\, \lambda_{\,1}}{1 \,+\, \lambda_{\,2}}\,\right)^{\,2}\; \|\,f\,\|^{\,2} \;\leq\; \sum\limits_{\,j \,\in\, J}\,w_{j}^{\,2}\, \left\|\,\Lambda_{j}\,P_{\,W_{j}}\,(\,f\,) \,\right\|^{\,2}.\hspace{1.1cm}
\end{equation}
Thus, \,$\Lambda$\; is a \,$g$-fusion frame for \,$H$\, with bounds \,$\dfrac{1}{D_{\,2}}\;\left(\,\dfrac{1 \,-\, \lambda_{\,1}}{1 \,+\, \lambda_{\,2}}\,\right)^{\,2}$\, and \,$D_{\,1}$.
\end{proof}

\begin{theorem}\label{th4}
Let \,$S_{\Gamma\,\Lambda}$\, be the frame operator for a pair of g-fusion Bessel sequences \,$\Lambda$\, and \,$\Gamma$\, of bounds \,$D_{\,1}$\, and \,$D_{\,2}$, respectively.\;Assume \;$\lambda \,\in\, [\,0 \,,\, 1\,)$\; such that 
\[\left\|\,f \,-\, S_{\Gamma\,\Lambda}\,f\,\right\| \;\leq\; \lambda\, \|\,f\,\|\;  \;\;\forall\; f \,\in\, H.\]
Then \,$\Lambda\; \;\text{and}\; \;\Gamma$\; are \,$g$-fusion frames for \,$H$. 
\end{theorem}

\begin{proof}
By putting \;$\lambda_{\,1} \,=\, \lambda\; \;\text{and}\; \; \lambda_{\,2} \,=\, 0$\; in (\ref{eq7}), we get 
\[ \dfrac{(\,1 \,-\, \lambda\,)^{\,2}}{D_{\,2}}\; \|\,f\,\|^{\,2} \;\leq\; \sum\limits_{\,j \,\in\, J}\,w_{\,j}^{\,2}\, \left\|\,\Lambda_{\,j}\,P_{\,W_{\,j}}\,(\,f\,) \,\right\|^{\,2}\]and therefore \,$\Lambda$\; is a \,$g$-fusion frame.\;Now, for each \,$f \,\in\, H$, we have 
\[\left\|\,f \,-\, S_{\,\Gamma\,\Lambda}^{\,\ast}\,f\,\right\| \,=\,\left\|\,\left(\,I_{\,H} \,-\, S_{\,\Gamma\,\Lambda}\,\right)^{\,\ast}\,f\,\right\| \,\leq\, \left\|\,\left(\,I_{\,H} \,-\, S_{\,\Gamma\,\Lambda}\,\right)\,\right\|\; \|\,f\,\| \,\leq\, \lambda\; \|\,f\,\|\]
\[\Rightarrow\, \left(\,1 \,-\, \lambda\,\right)\,\|\,f\,\| \,\leq\, \|\,S_{\,\Gamma\,\Lambda}^{\,\ast}\,f \,\| \,\leq\, \sqrt{D_{\,1}}\; \left(\,\sum\limits_{\,j \,\in\, J}\, v_{j}^{\,2}\; \left\|\,\Gamma_{j}\,P_{\,V_{j}}\,(\,f\,) \,\right\|^{\,2}\,\right)^{\,\dfrac{1}{2}}\; [\;\text{using (\ref{eq6})}\;]\]
\[\Rightarrow\; \sum\limits_{\,j \,\in\, J}\, v_{j}^{\,2}\; \left\|\,\Gamma_{j}\, P_{\,V_{j}}\,(\,f\,) \,\right\|^{\,2} \;\geq\; \dfrac{(\,1 \,-\, \lambda\,)^{\,2}}{D_{\,1}}\; \|\,f\,\|^{\,2}\; \;\;\forall\, f \,\in\, H.\] Hence, \,$\Gamma$\; is a \,$g$-fusion frame with bounds \,$\dfrac{(\,1 \,-\, \lambda\,)^{\,2}}{D_{\,1}}$\, and \,$D_{\,2}$.  
\end{proof}

\begin{definition}
Let \,$H$\; and \,$X$\; be two Hilbert spaces.\;Define
\[H\,\oplus\,X \,=\, \left\{\,(\,f \,,\, g\,) \;:\; f \,\in\, H,\; g \,\in\, X\,\right\}.\;\text{Then}\; \;H\,\oplus\,X\; \text{forms a}\]Hilbert space with respect to point-wise operations and inner product defined by  
\[\left<\,(\,f \,,\, g\,) \,,\, (\,f^{\,\prime} \,,\, g^{\,\prime}\,)\,\right> \,=\, \left<\,f \,,\, f^{\,\prime}\,\right>_{H} \,+\, \left<\,g \,,\, g^{\,\prime}\,\right>_{X}\; \;\forall\, f,\, f^{\,\prime} \,\in\, H\; \;\&\; \;\forall\, g,\, g^{\,\prime} \,\in\, X  .\]Now, if \,$U \,\in\, \mathcal{B}\,(\,H \,,\, Z\,),\, V \,\in\, \mathcal{B}\,(\,X\, \,,\, Y)$\, then for all \,$\;f \,\in\, H,\; g \,\in\, X$, we define 
\[U\,\oplus\,V \,\in\, \mathcal{B}\,\left(\,H\,\oplus\,X\, \,,\, Z\,\oplus\,Y \right)\; \;\text{by}\; \;\left(\,U\,\oplus\,V \,\right)\,(\,f \,,\, g\,) \,=\, \left(\,U\,f \,,\, V\,g\,\right),\]and 
$\left(\,U\,\oplus\,V\,\right)^{\,\ast} \,=\, U^{\,\ast}\,\oplus\,V^{\,\ast}$, where \,$Z,\, Y$\, are Hilbert spaces and also we define \,$P_{\,M\,\oplus\,\,N}\,(\,f \,,\, g\,) \,=\, \left(\,P_{\,M}\,f \,,\, P_{N}\,g \,\right)$, where \,$P_{\,M},\, P_{\,N}\; \;\text{and}\; \;P_{\,M\,\oplus\,\,N}$\; are the orthonormal projection onto the closed subspaces  \;$M \,\subset\, H,\, N \,\subset\, X \;\text{and}\; \;M\,\oplus\,\,N \,\subset\, H\,\oplus\,X$, respectively.  
\end{definition}

From here we assume that for each \,$j \,\in\, J,\, \,W_{j}\,\oplus\,V_{j}$\, are the closed subspaces of \,$H\,\oplus\,X$\, and \,$\Gamma_{j} \,\in\, \mathcal{B}\,(\,X \,,\, X_{j})$, where \,$\{\,X_{j}\,\}_{j \,\in\, J}$\, are the collection of Hilbert spaces and \,$\Lambda_{j}\,\oplus\,\Gamma_{j} \,\in\, \mathcal{B}\,\left(\,H\,\oplus\,X \,,\, H_{j} \,\oplus\, X_{j} \right)$.

\begin{theorem}
Let \,$\Lambda \,=\, \left\{\,\left(\,W_{j},\, \Lambda_{j},\, v_{j}\,\right)\,\right\}_{j \,\in\, J}$\; be a g-fusion frame for \,$H$\; with bounds \,$A,\, B$\; and \,$\Gamma \,=\, \left\{\,\left(\,V_{j},\, \Gamma_{j},\, v_{j}\,\right)\,\right\}_{j \,\in\, J}$\; be a g-fusion frame for \,$X$\; with bounds \,$C,\, D$.\;Then \,$\Lambda\,\oplus\,\Gamma \,=\, \left\{\,\left(\,W_{j}\,\oplus\,V_{j},\; \Lambda_{j}\,\oplus\,\Gamma_{j},\; v_{j}\,\right)\,\right\}_{j \,\in\, J}$\; is a g-fusion frame for \,$H\,\oplus\,X$\; with bounds \,$\min\,\{\,A,\, C\,\},\; \max\,\{\,B,\, D\,\}$.\;Furthermore, if \,$S_{\Lambda},\, S_{\Gamma}$\; and \,$S_{\Lambda\,\oplus\,\Gamma}$\; are g-fusion frame operators for \,$\Lambda,\, \Gamma$\; and \,$\Lambda\,\oplus\,\Gamma$\, respectively then we have \,$S_{\Lambda\,\oplus\,\Gamma} \,=\, S_{\Lambda}\,\oplus\,S_{\Gamma}$. 
\end{theorem}

\begin{proof}
Let \,$(\,f \,,\, g\,) \,\in\, H\,\oplus\,X$\; be an arbitrary element.\;Then
\[\sum\limits_{\,j \,\in\, J}\,v_{j}^{\,2}\; \left\|\,\left(\,\Lambda_{j}\,\oplus\,\Gamma_{j}\,\right)\, P_{\,W_{j}\,\oplus\,V_{j}}\,(\,f \,,\, g\,)\,\right\|^{\,2}\]
\[\,=\, \sum\limits_{\,j \,\in\, J}\,v_{j}^{\,2}\; \left<\,\left(\,\Lambda_{j}\,\oplus\,\Gamma_{j}\,\right)\, P_{\,W_{j}\,\oplus\,V_{j}}\,(\,f \,,\, g\,)\, \;,\; \left(\,\Lambda_{j}\,\oplus\,\Gamma_{j}\,\right)\, P_{\,W_{j}\,\oplus\,V_{j}}\,(\,f \,,\, g\,) \,\right>\hspace{1.2cm}\]
\[=\; \sum\limits_{\,j \,\in\, J}\,v_{j}^{\,2}\;\left<\,\Lambda_{j}\,\oplus\,\Gamma_{j}\, \left(\,P_{\,W_{j}}\,(\,f\,) \,,\, \,P_{\,V_{j}}\,(\,g\,)\,\right) \;,\; \Lambda_{j}\,\oplus\,\Gamma_{j}\,\left(\,P_{\,W_{j}}\,(\,f\,) \,,\, \,P_{\,V_{j}}\,(\,g\,)\,\right)\,\right>\]
\[=\; \sum\limits_{\,j \,\in\, J}\,v_{j}^{\,2}\;\left<\,\left(\,\Lambda_{j}\,P_{\,W_{j}}\,(\,f\,) \,,\, \Gamma_{j}\,P_{\,V_{j}}\,(\,g\,)\,\right) \;,\; \,\left(\,\Lambda_{j}\,P_{\,W_{j}}\,(\,f\,) \,,\, \,\Gamma_{j}\,P_{\,V_{j}}\,(\,g\,)\,\right)\,\right> \hspace{1.2cm}\]
\[=\; \sum\limits_{\,j \,\in\, J}\,v_{j}^{\,2}\; \left(\,\left<\,\Lambda_{j}\,P_{\,W_{j}}\,(\,f\,) \;,\; \Lambda_{j}\,P_{\,W_{j}}\,(\,f\,)\,\right>_{H} \,+\, \left<\,\Gamma_{j}\,P_{\,V_{j}}\,(\,g\,) \;,\; \Gamma_{j}\,P_{\,V_{j}}\,(\,g\,)\,\right>_{X}\,\right)\hspace{.2cm}\]
\[=\; \sum\limits_{\,j \,\in\, J}\,v_{j}^{\,2}\; \left(\,\left\|\,\Lambda_{j}\,P_{\,W_{j}}\,(\,f\,)\,\right\|_{H}^{\,2} \;+\; \left\|\,\Gamma_{j}\,P_{\,V_{j}}\,(\,g\,)\,\right\|_{X}^{\,2}\,\right)\hspace{4.4cm}\]
\[ \,=\, \sum\limits_{\,j \,\in\, J}\,v_{j}^{\,2}\; \left\|\,\Lambda_{j}\,P_{\,W_{j}}\,(\,f\,)\,\right\|_{H}^{\,2} \;+\; \sum\limits_{\,j \,\in\, J}\,v_{j}^{\,2}\; \left\|\,\Gamma_{j}\,P_{\,V_{j}}\,(\,g\,)\,\right\|_{X}^{\,2}\hspace{3.7cm}\] 
\[\hspace{.6cm}\leq\; B\; \|\,f\,\|_{H}^{\,2} \;+\; D\; \|\,g\,\|_{X}^{\,2}\; \;[\;\text{since}\; \Lambda,\; \Gamma\; \;\text{are $g$-fusion frames}\;]\hspace{3.9cm}\]
\[\leq\; \max\,\{\,B,\, D\,\}\; \left(\,\|\,f\,\|_{H}^{\,2} \;+\; \|\,g\,\|_{X}^{\,2}\,\right) \,=\, \max\,\{\,B,\, D\,\}\; \|\,(\,f \,,\, g\,)\,\|^{\,2}.\hspace{1.9cm}\]Similarly, it can be shown that
\[\min\,\{\,A,\, C\,\}\; \|\,(\,f \,,\, g\,)\,\|^{\,2}\; \leq\; \sum\limits_{\,j \,\in\, J}\,v_{j}^{\,2}\; \left\|\,\left(\,\Lambda_{j}\,\oplus\,\Gamma_{j}\,\right)\, P_{\,W_{j}\,\oplus\,V_{j}}\,(\,f \,,\, g\,)\,\right\|^{\,2}.\]Therefore, for all \,$(\,f \,,\, g\,) \,\in\, H\,\oplus\,X$, we have
\[A_{\,1}\, \|\,(\,f \,,\, g\,)\,\|^{\,2}\, \leq\, \sum\limits_{\,j \,\in\, J}\,v_{j}^{\,2}\, \left\|\,(\,\Lambda_{j}\,\oplus\,\Gamma_{j}\,)\, P_{\,W_{j}\,\oplus\,V_{j}}\,(\,f \,,\, g\,)\,\right\|^{\,2} \,\leq\, B_{\,1}\, \|\,(\,f \,,\, g\,)\,\|^{\,2}\]
and hence \,$\Lambda\,\oplus\,\Gamma$\, is a \,$g$-fusion frame for \,$H\,\oplus\,X$\, with bounds \,$A_{\,1} \,=\, \min\,\{\,A,\, C\,\}$\; and \,$B_{\,1} \,=\, \max\,\{\,B,\, D\,\}$\,.\;Furthermore, for \,$(\,f \,,\, g\,) \,\in\, H\,\oplus\,X$, we have 
\[S_{\Lambda\,\oplus\,\Gamma}\,(\,f \,,\, g\,) \,=\, \sum\limits_{\,j \,\in\, J}\,v_{j}^{\,2}\; P_{\,W_{j}\,\oplus\,V_{j}}\, \left(\,\Lambda_{j}\,\oplus\,\Gamma_{j}\,\right)^{\,\ast}\, \left(\,\Lambda_{j}\,\oplus\,\Gamma_{j}\,\right)\, P_{\,W_{j}\,\oplus\,V_{j}}\,(\,f \,,\, g\,)\hspace{1.5cm}\]
\[\hspace{1.9cm}=\; \sum\limits_{\,j \,\in\, J}\,v_{j}^{\,2}\; P_{\,W_{j}\,\oplus\,V_{j}}\, \left(\,\Lambda_{j}\,\oplus\,\Gamma_{j}\,\right)^{\,\ast}\, \left(\,\Lambda_{j}\,\oplus\,\Gamma_{j}\,\right)\, \left(\,P_{\,W_{j}}\,(\,f\,) \;,\; P_{\,V_{j}}\,(\,g\,)\,\right)\] 
\[\hspace{1cm} =\; \sum\limits_{\,j \,\in\, J}\,v_{j}^{\,2}\; P_{\,W_{j}\,\oplus\,V_{j}}\, \left(\,\Lambda_{j}\,\oplus\,\Gamma_{j}\,\right)^{\,\ast}\, \left(\,\Lambda_{j}\,P_{\,W_{j}}\,(\,f\,) \;,\; \Gamma_{j}\,P_{\,V_{j}}\,(\,g\,)\,\right)\]
\[\hspace{1cm}=\; \sum\limits_{\,j \,\in\, J}\,v_{j}^{\,2}\; P_{\,W_{j}\,\oplus\,V_{j}}\, \left(\,\Lambda_{j}^{\,\ast}\,\oplus\,\Gamma_{j}^{\,\ast}\,\right)\, \left(\,\Lambda_{j}\,P_{\,W_{j}}\,(\,f\,) \;,\; \Gamma_{j}\,P_{\,V_{j}}\,(\,g\,)\,\right)\]
\[ =\; \sum\limits_{\,j \,\in\, J}\,v_{j}^{\,2}\; P_{\,W_{j}\,\oplus\,V_{j}}\, \left(\,\Lambda_{j}^{\,\ast}\, \Lambda_{j}\, P_{\,W_{j}}\,(\,f\,) \;,\; \Gamma_{j}^{\,\ast}\, \Gamma_{j}\, P_{\,V_{j}}\,(\,g\,)\,\right)\hspace{.1cm}\]
\[=\; \sum\limits_{\,j \,\in\, J}\,v_{j}^{\,2}\; \left(\,P_{\,W_{j}}\,\Lambda_{j}^{\,\ast}\, \Lambda_{j}\, P_{\,W_{j}}\,(\,f\,) \;,\; P_{\,V_{j}}\, \Gamma_{j}^{\,\ast}\, \Gamma_{j}\,P_{\,V_{j}}\,(\,g\,)\,\right)\hspace{.1cm}\]
\[\hspace{1.6cm}=\; \left(\,\sum\limits_{\,j \,\in\, J}\,v_{j}^{\,2}\; P_{\,W_{j}}\, \Lambda_{j}^{\,\ast}\, \Lambda_{j}\, P_{\,W_{j}}\,(\,f\,) \;,\; \sum\limits_{\,j \,\in\, J}\,v_{j}^{\,2}\, P_{\,V_{j}}\,\Gamma_{j}^{\,\ast}\, \Gamma_{j}\, P_{\,V_{j}}\,(\,g\,)\,\right)\]
\[\hspace{2.2cm}=\; \left(\,S_{\Lambda}\,(\,f\,) \;,\; S_{\Gamma}\,(\,g\,)\,\right) \,=\, \left(\,S_{\Lambda}\,\oplus\,S_{\Gamma}\,\right)\,(\,f \,,\, g\,)\; \;\forall\; (\,f \,,\, g\,)\,\in\, H\,\oplus\,X.\]
Hence, \,$S_{\Lambda\,\oplus\,\Gamma} \,=\, S_{\Lambda}\,\oplus\,S_{\Gamma}$.\;This completes the proof.    
\end{proof}

\begin{theorem}
Let \,$\Lambda\,\oplus\,\Gamma \,=\, \left\{\,\left(\,W_{j}\,\oplus\,V_{j},\; \Lambda_{j}\,\oplus\,\Gamma_{j},\; v_{j}\,\right)\,\right\}_{j \,\in\, J}$\; be a g-fusion frame for \,$H\,\oplus\,X$\; with frame operator \,$S_{\Lambda\,\oplus\,\Gamma}$.\;Then 
\[\Delta^{\,\prime} \,=\, \left\{\,\left(\,S_{\Lambda\,\oplus\,\Gamma}^{\,-\, \dfrac{1}{2}}\,\left(\,W_{j}\,\oplus\,V_{j}\,\right),\; (\,\Lambda_{j}\,\oplus\,\Gamma_{j}\,)\, P_{\,W_{j}\,\oplus\,V_{j}}\,S_{\Lambda\,\oplus\,\Gamma}^{\,-\, \dfrac{1}{2}},\; v_{j}\,\right)\,\right\}_{j \,\in\, J}\]is a Parseval g-fusion frame for \,$H\,\oplus\,X$. 
\end{theorem}

\begin{proof}
Since \,$S_{\Lambda\,\oplus\,\Gamma}$\, is a positive operator, there exists  a unique positive square root \,$S_{\Lambda\,\oplus\,\Gamma}^{\dfrac{1}{2}}\; \left(\text{or} \;S_{\Lambda\,\oplus\,\Gamma}^{\,-\, \dfrac{1}{2}}\,\right)$\, and they commute with \,$S_{\Lambda\,\oplus\,\Gamma}$\, and \,$S_{\Lambda\,\oplus\,\Gamma}^{\,-\, 1}$.\;Therefore, each \,$(\,f \,,\, g\,) \,\in\, H \,\oplus\, X$, can be written as \,$(\,f \,,\, g\,) \,=\, S_{\Lambda\,\oplus\,\Gamma}^{\,-\, \dfrac{1}{2}}\,S_{\Lambda\,\oplus\,\Gamma}\,S_{\Lambda\,\oplus\,\Gamma}^{\,-\, \dfrac{1}{2}}\,(\,f \,,\, g\,)$
\[=\, \sum\limits_{\,j \,\in\, J}\,v_{j}^{\,2}\,S_{\Lambda\,\oplus\,\Gamma}^{\,-\, \dfrac{1}{2}}\, P_{\,W_{j}\,\oplus\,V_{j}}\, \left(\,\Lambda_{j}\,\oplus\,\Gamma_{j}\,\right)^{\,\ast}\, (\,\Lambda_{j}\,\oplus\,\Gamma_{j}\,)\, P_{\,W_{j}\,\oplus\,V_{j}}\,S_{\Lambda\,\oplus\,\Gamma}^{\,-\, \dfrac{1}{2}}\,(\,f \,,\, g\,).\]Now, for each \,$(\,f \,,\, g\,) \,\in\, H \,\oplus\, X$, we have
\[\|\,(\,f \,,\, g\,)\,\|^{\,2} \,=\, \left<\,(\,f \,,\, g\,) \,,\, (\,f \,,\, g\,)\,\right>\]
\[=\, \left<\,\sum\limits_{\,j \,\in\, J}\,v_{j}^{\,2}\,S_{\Lambda\,\oplus\,\Gamma}^{\,-\, \dfrac{1}{2}}\, P_{\,W_{j}\,\oplus\,V_{j}}\, \left(\,\Lambda_{j}\,\oplus\,\Gamma_{j}\,\right)^{\,\ast}\, (\,\Lambda_{j}\,\oplus\,\Gamma_{j}\,)\, P_{\,W_{j}\,\oplus\,V_{j}}\,S_{\Lambda\,\oplus\,\Gamma}^{\,-\, \dfrac{1}{2}}\,(\,f \,,\, g\,) \;,\; (\,f \,,\, g\,)\,\right>\]
\[=\, \sum\limits_{\,j \,\in\, J}\,v_{j}^{\,2}\,\left<\,\left(\,\Lambda_{j}\,\oplus\,\Gamma_{j}\,\right)\, P_{\,W_{j}\,\oplus\,V_{j}}\,S_{\Lambda\,\oplus\,\Gamma}^{\,-\, \dfrac{1}{2}}\,(\,f \,,\, g\,) \;,\; \left(\,\Lambda_{j}\,\oplus\,\Gamma_{j}\,\right)\, P_{\,W_{j}\,\oplus\,V_{j}}\,S_{\Lambda\,\oplus\,\Gamma}^{\,-\, \dfrac{1}{2}}\,(\,f \,,\, g\,)\,\right>\]
\[=\, \sum\limits_{\,j \,\in\, J}\,v_{j}^{\,2}\,\left\|\,\left(\,\Lambda_{j}\,\oplus\,\Gamma_{j}\,\right)\, P_{\,W_{j}\,\oplus\,V_{j}}\,S_{\Lambda\,\oplus\,\Gamma}^{\,-\, \dfrac{1}{2}}\,(\,f \,,\, g\,)\,\right\|^{\,2}\hspace{5.5cm}\]
\[=\, \sum\limits_{\,j \,\in\, J}\,v_{j}^{\,2}\,\left\|\,\left(\,\Lambda_{j}\,\oplus\,\Gamma_{j}\,\right)\, P_{\,W_{j}\,\oplus\,V_{j}}\,S_{\Lambda\,\oplus\,\Gamma}^{\,-\, \dfrac{1}{2}}\;P_{\left(S_{\Lambda\,\oplus\,\Gamma}^{\,-\, \dfrac{1}{2}}\,(\,W_{j}\,\oplus\,V_{j}\,)\right)}\,(\,f \,,\, g\,)\,\right\|^{\,2}\; [\;\text{by Theorem (\ref{th1.01})}\;].\]This shows that \,$\Delta^{\,\prime}$\, is a Parseval \,$g$-fusion frame for \,$H \,\oplus\, X$.     
\end{proof}

\begin{theorem}\label{th5}
Let \,$\Lambda\,\oplus\,\Gamma \,=\, \left\{\,\left(\,W_{j}\,\oplus\,V_{j},\; \Lambda_{j}\,\oplus\,\Gamma_{j},\; v_{j}\,\right)\,\right\}_{j \,\in\, J}$\; be a g-fusion frame for \,$H\,\oplus\,X$\; with bounds \,$A_{\,1},\, B_{\,1}$\; and \,$S_{\Lambda\,\oplus\,\Gamma}$\, be the corresponding frame operator.\;Then 
\[\Delta \,=\, \left\{\,\left(\,S_{\Lambda\,\oplus\,\Gamma}^{\,-\, 1}\,\left(\,W_{j}\,\oplus\,V_{j}\,\right),\; \left(\,\Lambda_{j}\,\oplus\,\Gamma_{j}\,\right)\, P_{\,W_{j}\,\oplus\,V_{j}}\,S_{\Lambda\,\oplus\,\Gamma}^{\,-\, 1},\; v_{j}\,\right)\,\right\}_{j \,\in\, J}\] is a g-fusion frame for \,$H\,\oplus\,X$\; with frame operator \,$S_{\Lambda\,\oplus\,\Gamma}^{\,-\, 1}$.
\end{theorem}

\begin{proof}
For any \,$(\,f \,,\, g\,) \,\in\, H\,\oplus\,X$, we have 
\[(\,f \,,\, g\,) \,=\, S_{\Lambda\,\oplus\,\Gamma}\,S_{\Lambda\,\oplus\,\Gamma}^{\,-\, 1}\,(\,f \,,\, g\,)\hspace{9cm}\]
\begin{equation}\label{eq10}
\hspace{1.2cm}\,=\, \sum\limits_{\,j \,\in\, J}\,v_{j}^{\,2}\, P_{\,W_{j}\,\oplus\,V_{j}}\, \left(\,\Lambda_{j}\,\oplus\,\Gamma_{j}\,\right)^{\,\ast}\, (\,\Lambda_{j}\,\oplus\,\Gamma_{j}\,)\, P_{\,W_{j}\,\oplus\,V_{j}}\,S_{\Lambda\,\oplus\,\Gamma}^{\,-\, 1}\,(\,f \,,\, g\,).
\end{equation}
By Theorem (\ref{th1.01}), for any \,$(\,f \,,\, g\,) \,\in\, H\,\oplus\,X$, we have
\[\sum\limits_{\,j \,\in\, J}\,v_{j}^{\,2}\,\left\|\,\left(\,\Lambda_{j}\,\oplus\,\Gamma_{j}\,\right)\, P_{\,W_{j}\,\oplus\,V_{j}}\,S_{\Lambda\,\oplus\,\Gamma}^{\,-\, 1}\; P_{\,S_{\Lambda\,\oplus\,\Gamma}^{\,-\, 1}\,\left(\,W_{j}\,\oplus\,V_{j}\,\right)}\,(\,f \,,\, g\,)\,\right\|^{\,2}\]
\begin{equation}\label{eq11}
\hspace{1cm}\,=\, \sum\limits_{\,j \,\in\, J}\,v_{j}^{\,2}\,\left\|\,\left(\,\Lambda_{j}\,\oplus\,\Gamma_{j}\,\right)\, P_{\,W_{j}\,\oplus\,V_{j}}\,S_{\Lambda\,\oplus\,\Gamma}^{\,-\, 1}\,(\,f \,,\, g\,)\,\right\|^{\,2}
\end{equation}
\[\hspace{3.42cm}\leq\, B_{\,1}\, \left\|\,S_{\Lambda\,\oplus\,\Gamma}^{\,-\, 1}\,\right\|^{\,2}\,\|\,(\,f \,,\, g\,)\,\|^{\,2}\; \;[\;\text{since}\; \Lambda\,\oplus\,\Gamma\; \text{is $g$-fusion frame}\;].\] 
On the other hand, using (\ref{eq10}), we get
\[\|\,(\,f \,,\, g\,)\,\|^{\,4} \,=\, \left|\,\left<\,(\,f \,,\, g\,) \,,\, (\,f \,,\, g\,)\,\right>\,\right|^{\,2}\]
\[\,=\, \left|\,\left<\,\sum\limits_{\,j \,\in\, J}\,v_{j}^{\,2}\, P_{\,W_{j}\,\oplus\,V_{j}}\, \left(\,\Lambda_{j}\,\oplus\,\Gamma_{j}\,\right)^{\,\ast}\, (\,\Lambda_{j}\,\oplus\,\Gamma_{j}\,)\, P_{\,W_{j}\,\oplus\,V_{j}}\, S_{\Lambda\,\oplus\,\Gamma}^{\,-\, 1}\,(\,f \,,\, g\,) \;,\; (\,f \,,\, g\,)\,\right>\,\right|^{\,2}\]
\[=\, \left|\,\sum\limits_{\,j \,\in\, J}\,v_{j}^{\,2}\,\left<\,(\,\Lambda_{j}\,\oplus\,\Gamma_{j}\,)\; P_{\,W_{j}\,\oplus\,V_{j}}\,S_{\Lambda\,\oplus\,\Gamma}^{\,-\, 1}\,(\,f \,,\, g\,) \;,\; \left(\,\Lambda_{j}\,\oplus\,\Gamma_{j}\,\right)\,P_{\,W_{j}\,\oplus\,V_{j}}\;(\,f \,,\, g\,)\,\right>\,\right|^{\,2}\hspace{1.5cm}\]
\[\leq\, \sum\limits_{\,j \,\in\, J}\,v_{j}^{\,2}\; \left\|\,(\,\Lambda_{j}\,\oplus\,\Gamma_{j}\,)\; P_{\,W_{j}\,\oplus\,V_{j}}\,S_{\Lambda\,\oplus\,\Gamma}^{\,-\, 1}\,(\,f \,,\, g\,)\,\right\|^{\,2}\, \sum\limits_{\,j \,\in\, J}\,v_{j}^{\,2}\; \left\|\,(\,\Lambda_{j}\,\oplus\,\Gamma_{j}\,)\; P_{\,W_{j}\,\oplus\,V_{j}}\;(\,f \,,\, g\,)\,\right\|^{\,2}\]
\[\leq\, \sum\limits_{\,j \,\in\, J}\,v_{j}^{\,2}\; \left\|\,(\,\Lambda_{j}\,\oplus\,\Gamma_{j}\,)\; P_{\,W_{j}\,\oplus\,V_{j}}\,S_{\Lambda\,\oplus\,\Gamma}^{\,-\, 1}\,(\,f \,,\, g\,)\,\right\|^{\,2}\, B_{\,1}\,\|\,(\,f \,,\, g\,)\,\|^{\,2}\; [\,\text{as}\; \Lambda \,\oplus\, \Gamma \;\text{is $g$-fusion frame}\;]\]
\[=\, \,B_{\,1}\,\|\,(\,f \,,\, g\,)\,\|^{\,2}\, \sum\limits_{\,j \,\in\, J}\,v_{j}^{\,2}\,\left\|\,(\,\Lambda_{j}\,\oplus\,\Gamma_{j}\,)\, P_{\,W_{j}\,\oplus\,V_{j}}\,S_{\Lambda\,\oplus\,\Gamma}^{\,-\, 1}\,P_{\,S_{\Lambda\,\oplus\,\Gamma}^{\,-\, 1}\,\left(\,W_{j}\,\oplus\,V_{j}\,\right)}\,(\,f \,,\, g\,)\,\right\|^{\,2} [\;\text{from (\ref{eq11})}\;].\]
Therefore,
\[ B_{\,1}^{\,-\, 1}\, \|\,(\,f \,,\, g\,)\,\|^{\,2} \,\leq\, \sum\limits_{\,j \,\in\, J}\,v_{j}^{\,2}\,\left\|\,(\,\Lambda_{j}\,\oplus\,\Gamma_{j}\,)\, P_{\,W_{j}\,\oplus\,V_{j}}\,S_{\Lambda\,\oplus\,\Gamma}^{\,-\, 1}\,P_{\,S_{\Lambda\,\oplus\,\Gamma}^{\,-\, 1}\,\left(\,W_{j}\,\oplus\,V_{j}\,\right)}\,(\,f \,,\, g\,)\,\right\|^{\,2}.\]
Hence, \,$\Delta$\; is a \,$g$-fusion frame for \,$H\,\oplus\,X$.\;Let \,$S_{\Delta}$\; be the \,$g$-fusion frame operator for \,$\Delta$\, and take \,$\Delta_{j} \,=\, \Lambda_{j}\,\oplus\,\Gamma_{j}$.\;Now, for each \,$(\,f \,,\, g\,) \,\in\, H\,\oplus\,X,\; S_{\Delta}\,(\,f \,,\, g\,)$
\[=\sum\limits_{\,j \,\in\, J}\,v_{j}^{\,2}\, P_{\,S_{\Lambda\,\oplus\,\Gamma}^{\,-\, 1} \left(\,W_{j}\,\oplus\,V_{j}\,\right)} \left(\,\Delta_{j}\, P_{\,W_{j}\,\oplus\,V_{j}}\,S_{\Lambda\,\oplus\,\Gamma}^{\,-\, 1}\,\right)^{\,\ast} \left(\,\Delta_{j}\,P_{\,W_{j}\,\oplus\,V_{j}}\,S_{\Lambda\,\oplus\,\Gamma}^{\,-\, 1}\,\right) P_{\,S_{\Lambda\,\oplus\,\Gamma}^{\,-\, 1}\,\left(\,W_{j}\,\oplus\,V_{j}\,\right)} (\,f \,,\, g\,)\]
\[= \sum\limits_{\,j \,\in\, J}\,v_{j}^{\,2} \left(\,P_{\,W_{j}\,\oplus\,V_{j}}\, S_{\Lambda\,\oplus\,\Gamma}^{\,-\, 1}\, P_{\,S_{\Lambda\,\oplus\,\Gamma}^{\,-\, 1} (\,W_{j}\,\oplus\,V_{j}\,)}\,\right)^{\,\ast} \Delta^{\,\ast}_{j}\, \Delta_{j}\left(\,P_{\,W_{j}\,\oplus\,V_{j}}\,S_{\Lambda\,\oplus\,\Gamma}^{\,-\, 1}\, P_{\,S_{\Lambda\,\oplus\,\Gamma}^{\,-\, 1} (\,W_{j}\,\oplus\,V_{j}\,)}\,\right)(\,f \,,\, g\,)\]
\[=\, \sum\limits_{\,j \,\in\, J}\,v_{j}^{\,2}\, \left(\,P_{\,W_{j}\,\oplus\,V_{j}}\, S_{\Lambda\,\oplus\,\Gamma}^{\,-\, 1}\,\right)^{\,\ast}\, \Delta^{\,\ast}_{j}\, \Delta_{j}\,\left(\,P_{\,W_{j}\,\oplus\,V_{j}}\,S_{\Lambda\,\oplus\,\Gamma}^{\,-\, 1} \,\right)\,(\,f \,,\, g\,)\; [\;\text{using Theorem (\ref{th1.01})}\;]\]
\[=\, \sum\limits_{\,j \,\in\, J}\,v_{j}^{\,2}\; S_{\Lambda\,\oplus\,\Gamma}^{\,-\, 1}\, P_{\,W_{j}\,\oplus\,V_{j}}\, \left(\,\Lambda_{j}\,\oplus\,\Gamma_{j}\,\right)^{\,\ast}\, \left(\,\Lambda_{j}\,\oplus\,\Gamma_{j}\,\right)\, \left(\,P_{\,W_{j}\,\oplus\,V_{j}}\,S_{\Lambda\,\oplus\,\Gamma}^{\,-\, 1} \,\right)\,(\,f \,,\, g\,)\hspace{2cm}\]
\[=\, S_{\Lambda\,\oplus\,\Gamma}^{\,-\, 1}\, \left(\,\sum\limits_{\,j \,\in\, J}\,v_{j}^{\,2}\; P_{\,W_{j}\,\oplus\,V_{j}}\, \left(\,\Lambda_{j}\,\oplus\,\Gamma_{j}\,\right)^{\,\ast}\, \left(\,\Lambda_{j}\,\oplus\,\Gamma_{j}\,\right)\, P_{\,W_{j}\,\oplus\,V_{j}}\,\,\left(\,S_{\Lambda\,\oplus\,\Gamma}^{\,-\, 1}\, (\,f \,,\, g\,)\,\right)\,\right)\hspace{.3cm}\]
\[=\, S_{\Lambda\,\oplus\,\Gamma}^{\,-\, 1}\; S_{\Lambda\,\oplus\,\Gamma}\, \left(\,S_{\Lambda\,\oplus\,\Gamma}^{\,-\, 1}\, (\,f \,,\, g\,)\,\right)\; \;[\;\text{by definition of }\;S_{\Lambda\,\oplus\,\Gamma}\;]. \hspace{5cm}\]
\[ \,=\, S_{\Lambda\,\oplus\,\Gamma}^{\,-\, 1}\, (\,f \,,\, g\,). \hspace{11.5cm}\]
Thus, \,$S_{\Delta} \,=\, S_{\Lambda\,\oplus\,\Gamma}^{\,-\, 1}$.\;This completes the proof.    
\end{proof}

\begin{note}
Form Theorem (\ref{th5}), we can conclude that if \,$\Lambda \,\oplus\, \Gamma$\, is a \,$g$-fusion frame for \,$H \,\oplus\, K$, then \,$\Delta$\, is also a \,$g$-fusion frame for \,$H \,\oplus\, K$.\;The \,$g$-fusion frame \,$\Delta$\, is a called the canonical dual \,$g$-fusion frame of \,$\Lambda \,\oplus\, \Gamma$.  
\end{note}

\end{document}